\documentclass[12pt,authoryear,letterpaper,reqno, oneside]{amsart}
\usepackage[left=1.4in,right=1.4in,bottom=1.4in,top=1.40in]{geometry}
\usepackage{amsmath,amsfonts,amssymb}%
\usepackage{amsthm}
\usepackage{mathrsfs} 
\usepackage[usenames]{color}
\usepackage{setspace}
\usepackage{comment}
\usepackage{latexsym}

\usepackage{rotating}
\usepackage{verbatim}
\usepackage[usenames]{color}
\usepackage{xcolor}
\usepackage{hyperref}
\usepackage{url}
\usepackage{tikz,tikz-qtree,ifthen,cancel}
\usepackage{array,tikz-qtree,ifthen,cancel}
\usepackage{graphicx}
\usepackage{adjustbox}
\usepackage{amsthm,graphicx,tikz,appendix,tikz-qtree,ifthen,cancel}
\usetikzlibrary{calc,shapes,patterns,positioning}
\usepackage{mathrsfs}
\usepackage{pifont}
\usepackage{enumitem}
\usepackage{amssymb}
\usepackage{bm}
\usepackage{amsbsy}

\usepackage{vruler}

\newcounter{margincounter}

\catcode `@ = 11
\renewcommand\section{\@startsection{section}{1}{\z@}%
                                   {-3.5ex \@plus -1ex \@minus -.2ex}%
                                   {2.3ex \@plus.2ex}%
                                   {\normalfont\large\bfseries}} 
\renewcommand\subsection{\@startsection{subsection}{2}{\z@}%
                                     {-3.25ex\@plus -1ex \@minus -.2ex}%
                                     {1.5ex \@plus .2ex}%
                                     {\normalfont\bfseries}} 
\renewcommand\subsubsection{\@startsection{subsubsection}{2}{\z@}%
                                     {-3.25ex\@plus -1ex \@minus -.2ex}%
                                     {1.5ex \@plus .2ex}%
                                   {\normalfont\itshape}} 
\catcode `@ = 12

\def\la{\langle}\def\ra{\rangle}

\newcommand{\eop}{\bigstar}  

\newcommand{\AAA}{{\mathcal A}}
\newcommand{\BB}{{\mathcal B}}
\newcommand{\CC}{{\mathcal C}}

\newcommand{\KK}{{\mathcal K}}
\newcommand{\LL}{{\mathcal L}}
\newcommand{\MM}{{\mathcal M}}
\newcommand{\NN}{{\mathcal N}}

\newcommand{\fAA}{{\mathfrak A}}
\newcommand{\fBB}{{\mathfrak B}}
\newcommand{\fCC}{{\mathfrak C}}

\newcommand{\fEE}{{\mathfrak E}}

\newcommand{\fLL}{{\mathfrak L}}
\newcommand{\fMM}{{\mathfrak M}}
\newcommand{\fNN}{{\mathfrak N}}
\newcommand{\fOO}{{\mathfrak O}}

\newcommand{\triple}[3]{\langle\kern1pt#1 , \:#2 , 
      \:#3 \kern1pt\rangle }
\newcommand{\trpl}[3]{\langle\kern1pt#1 , \:#2 , 
      \:#3 \kern1pt\rangle }

\newcount\skewfactor
\def\mathunderaccent#1#2 {\let\theaccent#1\skewfactor#2
\mathpalette\putaccentunder}
\def\putaccentunder#1#2{\oalign{$#1#2$\crcr\hidewidth
\vbox to.2ex{\hbox{$#1\skew\skewfactor\theaccent{}$}\vss}\hidewidth}}

\theoremstyle{plain}
\newtheorem{theorem}{Theorem}[section]
\newtheorem{proposition}[theorem]{Proposition}
\newtheorem{claim}[theorem]{Claim}
\newtheorem{lemma}[theorem]{Lemma}
\newtheorem*{ShLemma}{Lemma \ref{technical}}
\newtheorem{observation}[theorem]{Observation}
\newtheorem{corollary}[theorem]{Corollary}

\theoremstyle{definition}
\newtheorem{definition}[theorem]{Definition}
\newtheorem{question}[theorem]{Question}

\theoremstyle{remark}
\newtheorem{example}[theorem]{Example}
\newtheorem{fact}[theorem]{Fact.}
\newtheorem{notation}[theorem]{Notation}
\newtheorem{remark}[theorem]{Remark}

\newcommand{\sC}{\mathscr{C}}
\newcommand{\sI}{\mathscr{I}}

\newcommand{\Domp}{\mathrm{Dom}}
\newcommand{\intern}{\mathrm{int}}
\newcommand{\lraw}{\longrightarrow}
\newcommand{\lrawi}{\longrightarrow_\mathrm{int}}
\newcommand{\wh}[1]{\widehat{#1}}

\newcommand{\DD}{{\mathscr D}}

\newcommand{\K}{\mathcal{K}}
\newcommand{\raw}{\rightarrow}
\newcommand{\ov}{\overline}
\newcommand{\ran}{\textrm{ran}}
\newcommand{\los}{\L o\'{s}}

\newcommand{\age}{\mathrm{age}}

\newcommand{\uphp}{\upharpoonright}
\newcommand{\smf}{\smallfrown}

\begin{document}
\title{Big Ramsey Degrees in Ultraproducts of Finite Structures}

\author{Dana Barto\v{s}ov\'a}
\address{University of Florida, 1400 Stadium Road, Gainesville, FL 32601}
\email{dbartosova@ufl.edu}
\urladdr{https://people.clas.ufl.edu/dbartosova/}

\author{Mirna D\v zamonja}
\address{IRIF, CNRS et Universit{\'e} de Paris Cit{\'e}, B{\^a}timent Sophie Germain, Case courrier 7014,
8 Place Aurélie Nemours,
75205 Paris Cedex 13, France}
\email{mdzamonja@irif.fr}
\urladdr{https://www.logiqueconsult.eu}
 
\author{Rehana Patel}
\address{African Institute for Mathematical Sciences, M'bour-Thi\`{e}s, Senegal}
\email{rpatel@aims-senegal.org}
\urladdr{}

\author{Lynn Scow}
\address{California State University, San Bernardino
5500 University Parkway 
San Bernardino, CA 92407} 
\email{lscow@csusb.edu} 
\urladdr{https://www.csusb.edu/profile/lynn.scow}

\thanks{
This research was carried out as part of the American Institute of Mathematics (AIM) SQuaRE program. The authors thank AIM for their support. Dana Barto\v{s}ov\'a is supported by NSF grant No.~DMS-1953955 and NSF CAREER grant no.~DMS-2144118. Mirna D\v zamonja thanks the European Union's Horizon 2020 research and innovation program for funding under the Maria Sk{\l}odowska-Curie grant agreement No.~1010232.  She equally thanks l'Institut d'Histoire et de Philosophie des Sciences et des Techniques, Universit\' e Panth{\'e}on-Sorbonne, where she is an Associate Member, and the University of East Anglia, Norwich, UK, where she is a Visiting Professor. 
This material is partially based upon work supported by the National Science Foundation under Grant No.~DMS-1928930 while Rehana Patel and Lynn Scow participated in a program hosted by the Mathematical Sciences Research Institute in Berkeley, California, during the Summer 2022 semester.  Lynn Scow is supported by NSF grant No.~DMS-2246995 in completing part of this work.
}



\begin{abstract}
We develop a 
transfer principle of structural Ramsey theory from 
finite structures to 
ultraproducts.
We show that under certain mild conditions, when a class of finite structures has finite small Ramsey degrees, under the (Generalized) Continuum Hypothesis the ultraproduct has finite big Ramsey degrees for internal colorings.
The necessity of restricting to internal colorings is demonstrated 
by the example of the ultraproduct 
of finite linear orders. 
Under CH, this ultraproduct $\fLL^*$ has, as a spine, $\eta_1$, 
an uncountable analogue of the order type of rationals $\eta$.
Finite big Ramsey degrees for $\eta$ were exactly calculated by Devlin in \cite{Devlin}. 
It is 
immediate from \cite{Tod87} that $\eta_1$ fails to have finite big Ramsey degrees.
Moreover,
we extend Devlin's coloring to $\eta_1$ to show that it witnesses big Ramsey degrees of finite tuples in $\eta$ on every copy of $\eta$ in $\eta_1,$ and consequently in $\fLL^*$. This 
work gives additional confirmation
that ultraproducts are a suitable environment
for 
studying Ramsey properties of finite and infinite structures. 
\end{abstract}

\maketitle


\setcounter{page}{1}
\thispagestyle{empty}

\begin{small}
\renewcommand\contentsname{\!\!\!\!}
\setcounter{tocdepth}{3}
\tableofcontents
\end{small}

\newpage


\section{Introduction} This paper fits into a general framework of investigation of the transfer of combinatorial properties from a class of finite structures to a `limit' of that class. Various limits of such classes have appeared in the literature and have been extensively studied, including Fra{\"i}sss{\'e} limits, ultraproducts and graphons. Transfers of various combinatorial and logical properties have been considered, including first order properties of structures, graph invariants and model-theoretical classification. The research spans from the classical works such as \cite{Los}, to some more recent works that have already become classics \cite{Lovaszgraphons}, from mathematics, to computer sciences.
Here we shall be interested in the transfer of {\it Ramsey properties} from a class of finite structures to their {\it ultraproduct}.

Ramsey theory started with the celebrated work of Ramsey 
\cite{Ramseyth}, in particular the Finite Ramsey Theorem,
which states that for any natural numbers $r, s, k$, there exists a number $R(r,s,k)$, called the Ramsey number, such that for any coloring of $r$-tuples of a finite set with at least $R(r,s,k)$ elements into $s$ colors, there is a
monochromatic subset of size $k$. This theorem follows from the Infinite Ramsey Theorem, which states that for
any coloring of $r$-tuples of an infinite set into 
$s$ colors, there is an infinite monochromatic subset.\footnote{It is worthwhile noting that Ramsey's original work actually starts with the infinite case, stated in terms of classification of binary relations on a countably infinite set.}
Ramifications of this theorem in various directions have been overwhelmingly present in all areas of discrete mathematics (see for example the books \cite{MR795592}, \cite{stevoOCA} and a survey article \cite{Hajnal-Larson}), so much so that it would be an injustice to attempt to give a historical overview in this limited space. Suffices to say that Ramsey theory is now in the very core of discrete mathematics, combinatorial set theory and theoretical computer sciences, with important applications to other fields of mathematics, such as topology, Banach spaces, or operator algebras.

Much research has been done into possible generalizations of Ramsey's Theorem, be it to larger infinities such as in the book \cite{MR795592}, or with strengthenings of the conclusion to preserve not just the prescribed size of a set, but also some structure on it. We shall be interested in the latter direction, which is called {\em structural Ramsey theory}. Even if in some particular cases, mostly of certain classes of graphs, there are positive results in structural Ramsey theory, it turns out that in general, preserving the structure by monochromatic sets is highly non-trivial. For example, the last open problem from Erd\H{o}s' list dates from 1956 and asks to characterize countable ordinals $\alpha$ such that every coloring of the pairs from $\alpha$ into two colors, say red and blue, has either a blue subset of order type $\alpha$ or a red subset of order type 3 (see \cite{Hajnal-Larson}). A classical example due to Sierpi\'nski \cite{MR1556708}
shows that there is a coloring of the pairs of the set $\mathbb{Q}$ of the rationals into two colors, such that there is no monochromatic set of the order type of $\mathbb{Q}$.

In view of these negative results, recent research in structural Ramsey theory has concentrated on the notion of {\it Ramsey degrees}. The idea is that even though finding structured monochromatic sets might be difficult or impossible, in many situations one can measure the degree of that difficulty. For example, when it comes to the colorings of the $n$-tuples of the rationals, Laver proved that this degree is finite (unpublished). Building on this result, Devlin in his 1979 Ph.D. thesis (for an elegant combinatorial proof, see \cite{Vuksanovic}) showed that there is a close connection of these degrees with the known sequence of tangent numbers. Let us introduce a bit of notation in order to introduce this result.

\begin{notation}\label{structureRamsey}
Let $\LL$ be a (first-order) signature 
and let $\fMM$, $\fNN$ and $\fOO$ be $\LL$-structures. We write $\binom{\fNN}{\fMM}$ to denote the set of all substructures of $\fNN$ that are isomorphic to $\fMM$. A function $c : \binom{\fNN}{\fMM} \raw k$ for some natural number $k$ is called a coloring.
Given a coloring $c$ of $\binom{\fNN}{\fMM}$ and a substructure $\fNN'$ of $\fNN$, we say that $\binom{\fNN'}{\fMM}$ is \emph{$\ell$-chromatic by $c$} if $c$ takes at most $\ell$-many  values on $\binom{\fNN'}{\fMM}$.

For $k, \ell \in \omega$, the standard Erd\H{o}s-Rado style arrow notation for this
\[
    \fOO \lraw \bigl(\fNN \bigr)^\fMM_{k,\ell}
\]
asserts that for any $k$-coloring $c$ of the copies of $\fMM$ in $\fOO$ there exists a copy $\fNN'$ of $\fNN$ in $\fOO$ such that $\binom{\fNN'}{\fMM}$ is $\ell$-chromatic by $c$.
If $\fOO=\fNN$ and there is some $n$ such that $\fMM$ denotes the unstructured $n$-tuples of $\fNN$, we simply write $\fNN \lraw \bigl(\fNN \bigr)^n_{k,\ell}$.
\end{notation}

In this notation, Devlin's theorem is that for any natural numbers $n$ and $k$
\[
\eta\lraw (\eta)^n_{k,t_n},
\]
where $t_n$ is the $(2n-1)$-st tangent number (and it does not depend on $k$) and $\eta$ denotes the order type of the rational numbers. For example, for pairs $t_2=2$, so any coloring of the pairs of the rationals into two colors will have
a $2$-chromatic subset isomorphic to the rationals.

Generalizing this idea and the idea of Ramsey numbers, \cite{FOUCHE1997309} introduces the notion of (small) Ramsey degrees. 

\begin{definition}\label{smallRamseydeg}
For $\K$ a class of finite $\LL$-structures and $\AAA\in \K$, we say that $\AAA$ has {\bf finite (small) Ramsey degree in $\bm{\K}$} if there exists $\ell \in \omega$ satisfying the following property: For each $k \in \omega$ and each $\BB \in \K$  there exists $\CC \in \K$ such that the partition relation 
\[ 
    \CC \lraw \bigl(\BB \bigr)^\AAA_{k,\ell} \, 
\]
holds. The smallest such $\ell \in \omega$ (if it exists) is called the {\bf Ramsey degree of $\bm{\AAA}$ in $\bm{\KK}$} and denoted $\bm{t(\AAA,\K)}$. \footnote{When $t(\AAA, \KK) = 1$, we say that $\AAA$ is a \emph{Ramsey object in $\KK$} and if no finite $t(\AAA, \K)$ exists, we say that $\AAA$ has \emph{infinite Ramsey degree in $\K$}.}

The Ramsey degree analogue of Ramsey numbers in the case of infinite structures is called {\em big Ramsey degrees}. 
Namely, let 
$\fNN$ be an infinite and $\AAA$ a finite $\LL$-structure. Then $\AAA$ has {\bf finite big Ramsey degree in $\fNN$} if there exists $\ell \in \omega$ such that for each $k \in \omega$, the partition relation
\[ 
    \fNN \lraw \bigl(\fNN \bigr)^\AAA_{k,\ell} \, 
\]
holds. The smallest such $\ell \in \omega$ is called the {\bf big Ramsey degree of $\bm{\AAA}$ in $\bm{\fNN}$} and denoted $\bm{ T(\AAA,\fNN)}$, where we set $\bm{ T(\AAA,\fNN)} = \infty$ if no such $\ell \in \omega$ exists.
\end{definition}

The question that we address in this paper is a specific instance of the general question of the relationship that
exists between the big and small Ramsey degrees of a finite $\LL$-structure, and in particular if the infinite structure under question has been obtained as some sort of limit of a known class of finite structures. In the case of 
Fra{\" i}ss{\'e} limits many interesting results have been obtained; an excellent survey by one of the main contributors to the area is Dobrinen's ICM 2022 address \cite{NatashaICM2022}. We recall that Fra{\" i}ss{\'e} limits are countably infinite structures.

Going to uncountable limits of finite structures, a natural candidate to consider is an {\em ultraproduct}.
An ultraproduct of a sequence $(\MM_i)_{i\in\omega}$ of finite structures of the same signature is defined using a given ultrafilter $\DD$ on $\omega$. It is simply the Cartesian product 
$\prod_{i\in\omega} \MM_i$ reduced by the equivalence relation of being the same on ultrafilter-many coordinates. This is denoted by $\fMM^\ast=\prod_{i\in\omega} \MM_i/\DD$. An ultraproduct $\fMM^\ast$ of the type we described, if infinite, always has the size of the continuum, $2^{\aleph_0}$ (see Theorem \ref{sizec}).   
Ultraproducts form a classical object of study in both model theory and set theory, see for example \cite{C-K} and \cite{Kanamori}\footnote{A wide mathematical audience became enthusiastic about ultraproducts  after their appearance in the work and in the online blog \cite{Taoultra} by Tao.}. A fundamental property of ultraproducts is given by \los's theorem \cite{Los}, which states that any first order sentence true in ultrafilter-many $\MM_i$ is true in the ultraproduct. This property makes it very natural to ask which Ramsey properties from the sequence of finite structures carry to the ultraproduct (Ramsey properties, of course, are in general not first order.)

The {\bf main theorem} 
of this paper is Theorem \ref{fabulous}. It concerns the case of $\LL$ being a finite relational signature and the sequence 
$(\MM_i)_{i\in\omega}$ sufficiently `increasing' with respect to the ultrafilter $\DD$ (the technical term is $\DD$-trending). The 
ultrafilter is assumed not to contain singletons (a common technical assumption).
Let $\fNN$ be an $\LL$-structure of size at most $\aleph_1$,
whose connection with $(\MM_i)_{i\in\omega}$ is that the class $\KK$ of all finite substructures of $\fNN$,  called the {\em  age} of $\fNN$, is contained in the union of the ages of the $\MM_i$, for $i \in \omega$. Note, that by the known saturation properties of the ultraproduct, in this case
$\fNN$ is isomorphic to a substructure of $\fMM^\ast=\prod_{i\in\omega} \MM_i/\DD$.

In this setup, suppose that $\AAA\in\KK$ has finite small Ramsey degree $t$
in $\KK$. The corollary to the main theorem, Corollary \ref{CHcorollary}, answers the following natural question under CH: what can be said
about Ramsey properties of the copies of $\fNN$ within $\fMM^\ast$, with 
respect to $\AAA$? The answer is that the Ramsey degree remains bounded by $t$, but provided we restrict the colorings to the {\em internal} ones.

What does it mean for a coloring to be internal? The precise definition of this is given in Definition \ref{internalcolorcopies} but, roughly, these are the colorings that are defined on copies of $\AAA$ in $\fMM^\ast$ by the ultraproduct transfer of a sequence of $(c_i)_{i\in\omega}$, where each $c_i$ is a coloring of copies of $\AAA$ in $\MM_i$.  Colorings that are not internal will be called \emph{external}.

Having made sense of the formal statement of Theorem \ref{fabulous}, let us comment on its hypotheses. Both the hypothesis on the size of
$\fNN$ and on the colorings being internal are justified. The former because the saturation of the ultraproduct is limited to structures of
size $\aleph_1$ and, after all, the ultraproduct itself might have size
$\aleph_1$ if we happen to be in a model of CH. The latter assumption
of Theorem \ref{fabulous} is more complex to justify, but we do so in Section \ref{eta1}. Namely, Corollary \ref{maincounterexample2} gives a rather dramatic  counterexample to Theorem \ref{fabulous} under CH in the case when the coloring is external.

In addition to the main theorem, Theorem \ref{fabulous}, the paper explores two further directions. The first concerns the general case of internal colorings on ultraproducts with any infinite number of coordinates. Theorem \ref{fabulous_general} shows that, with slight changes in the proof, an analogue of
Theorem \ref{fabulous} holds in this general case. It allows us to have an explicit calculation of an upper bound for
internal Ramsey degrees under GCH, as is done in Corollary \ref{GCHcorollary}. In contrast with this direction which led us from the 
concrete to the abstract, in Section \ref{eta1} we go to the very concrete and explore the case of linear orders under CH.  We first show how Corollary \ref{maincounterexample} follows from results in \cite{Tod87}, and thus show how Corollary \ref{CHcorollary} fails for an ultraproduct of finite linear orders if the coloring is allowed to be external.
The result in Corollary \ref{maincounterexample2}, already mentioned above, is obtained after a more detailed analysis of the ultraproduct of finite linear orders under CH. We show that the coloring of finite tuples of the Cantor tree by Devlin embedding types as in 
\cite{Vuksanovic} and \cite{MR2603812} can be extended to $2^{<\omega_1}$ to witness the optimal failure of Corollary \ref{CHcorollary}.

The paper is organized as follows. Section \ref{prelim} gives general preliminaries. Further preliminaries, related to partition properties and internal colorings, are given in Section \ref{intercolo}. The Main Theorem \ref{fabulous} appears in Section \ref{omegacase} and its generalization to further cardinals, under GCH, is given in Section \ref{general case}. The case of linear orders under CH is presented in Section \ref{eta1}. Various future directions are mentioned in Section \ref{openproblems}. The paper finishes by Section \ref{Appendix}, which is an appendix that gives a self-contained proof of Lemma \ref{technical}. (In Section \ref{prelim} we quoted Lemma \ref{technical} from \cite{Sh-c} but the proof in \cite{Sh-c} was left to the reader.)

\section{Preliminaries}\label{prelim}

In this section, we introduce our notation and describe some of the basic model-theoretic ideas that we will use. For basic logical notions such as signatures, formulas, structures and models, we refer the reader to \cite{ho93} or \cite{ma02}. For more details on types, saturation and ultraproducts see \cite{C-K}, as well as \cite{Sh-c}.  As usual, Ord denotes the class of ordinals, and AC, CH and GCH denote the Axiom of Choice, the Continuum Hypothesis and the Generalized Continuum Hypothesis, respectively.  We use $\mathfrak c$ to denote the cardinality of the set $2^\omega$. 

\subsection{Model-theoretic basics}

\subsubsection*{Structures, Embeddings and Substructures}

We will use the convention that fraktur letters $\fAA$, $\fBB$, $\fMM$, $\fNN$, etc. denote arbitrary structures, finite or infinite, and reserve calligraphic letters $\AAA$, $\BB$, $\MM$, $\NN$, etc. to denote finite structures, with the exception of $\LL$ (possibly with decorations) which will always denote a signature, and $\K$ (possibly with subscripts) which will always denote a class of structures in some fixed signature. 

For structures denoted by $\fAA$ (or $\AAA$), the corresponding roman letter $A$ will denote the underlying set of $\fAA$ (or $\AAA$), and $|A|$  will denote the cardinality of $A$. The cardinality of a structure is defined to be the cardinality of its underlying set. We will assume that all structures have nonempty underlying set.

For the rest of this subsection, $\LL$ will denote a signature of arbitrary cardinality.

Let $\fAA$ and $\fBB$ be $\LL$-structures.  An {\em $\LL$-embedding} (or simply {\em embedding}) from $\fAA$ to $\fBB$ is an injective function $\sigma: A \to B$ which preserves the interpretations of all constant, relation and function symbols from $\LL$, i.e, such that 
\begin{itemize}

    \item
    for any constant symbol $c$ in $\LL$, $\sigma(c^\fAA) = c^\fBB$; 
    
    \item
    for any function symbol $f$ in $\LL$ of arity say $i$, and any $a_0, \ldots , a_{i-1} \in A$, 
    \[ 
        \sigma(f^{\fAA}(a_0, \ldots , a_{i-1})) =  f^{\fBB}(\sigma(a_0), \ldots , \sigma(a_{i-1})); \text{ \ and } 
    \]
    
    \item 
    for any relation symbol $R$ in $\LL$ of arity say $j$, and any $a_0, \ldots , a_{j-1} \in A$, 
    \[ 
        \fAA \vDash R^{\fAA}(a_0, \ldots , a_{j-1}) \text{ \ if and only if \ } \fBB \vDash R^{\fBB}(\sigma(a_0), \ldots , \sigma(a_{j-1})).
    \]
\end{itemize}
An {\em $\LL$-isomorphism} (or simply {\em isomorphism}) from $\fAA$ to $\fBB$ is a bijective embedding from $\fAA$ to $\fBB$. We write $\fAA \hookrightarrow \fBB$ when there exists an embedding from $\fAA$ to $\fBB$, and we write $\fAA \cong \fBB$ when there exists an isomorphism from $\fAA$ to $\fBB$. When $\fAA \cong \fBB$, we say that $\fAA$ is a {\em copy} of $\fBB$. An {\em $\LL$-automorphism} (or simply {\em automorphism}) of an $\LL$-structure $\fAA$ is an isomorphism from $\fAA$ to $\fAA$.

Two $\LL$-structures $\fAA$ and $\BB$ are {\em elementarily equivalent}, written $\fAA \equiv \fBB$, when $\fAA$ and $\fBB$ satisfy the same first-order $\LL$-sentences.

An $\LL$-structure $\fAA$ is a {\em substructure} of an $\LL$-structure $\fBB$ if $A \subseteq B$ and the inclusion map is an embedding. For an $\LL$-structure $\fBB$ and subset $X$ of $B$, the {\em substructure of $\fBB$ generated by $X$}, written $\langle X \rangle_\fBB$, is the smallest substructure of $\fBB$ whose underlying set contains $X$. (Note that there is a unique such substructure of $\fBB$.) When $\fAA := \langle X \rangle_\fBB$, we say that $\fAA$ is the substructure {\em induced} by $\fBB$ on $X$. A substructure $\fAA$ of an $\LL$-structure $\fBB$ is {\em finitely generated} if it is generated by some finite subset of $B$, and the structure $\fBB$ is {\em locally finite} if every finitely generated substructure of $\fBB$ is finite. 

\begin{fact}
When $\LL$ is relational, for any $\LL$-structure $\fBB$ and subset $X$ of $B$, the underlying set of the substructure $\langle X \rangle_\fBB$ is $X$ itself. Hence, any relational structure is locally finite.
\end{fact}

In subsequent sections, we will be considering a special kind of collection of finitely generated structures, called an {\em age}.

\begin{definition}
    For an $\LL$-structure $\fAA$, the {\bf age} of $\fAA$, written $\bm{\age(\fAA)}$, is the class of all finitely generated $\LL$-structures that embed into $\fAA$. 
\end{definition}
\noindent Note that the age of a nonempty structure will not be a set. 

The age of any structure has two important properties, the {\em Hereditary} and the {\em Joint Embedding} Property, which may be seen respectively as closure under certain substructure operations and under certain ``gluings'' via embeddings. These we now define.

\begin{definition}
    Let $\K$ be a class of $\LL$-structures. We say that $\K$ has the {\bf Hereditary Property} (abbreviated {\bf HP}) if, whenever $\fBB \in \K$ and $\fAA$ is a finitely generated  substructure of $\fBB$, we have $\fAA \in \K$. We say that $\K$ has the {\bf Joint Embedding Property} (abbreviated {\bf JEP}) if, whenever $\fAA, \fBB \in \K$, there exists $\fCC\in \K$ such that $\fAA \hookrightarrow \fCC$ and $\fBB \hookrightarrow \fCC$.
\end{definition}

\begin{remark}
It is immediate that the age of any $\LL$-structure is nonempty and has the HP and JEP. Fra{\"i}ssé (see \cite{shorter-model}) proved the converse in the countable case, that is, whenever a nonempty class $\K$ of finitely generated $\LL$-structures has only countably many pairwise non-isomorphic elements and has the HP and JEP, the class $\K$ is the age of some countable $\LL$-structure.
\end{remark}

Let $\LL$ and $\wh{\LL}$ be signatures such that $\LL \subseteq \wh{\LL}$, and let $\fAA$ be an $\LL$-structure and $\wh{\fAA}$ an $\wh{\LL}$-structure. Then $\fAA$ is the {\em $\LL$-reduct} (or simply {\em reduct}) of $\wh{\fAA}$, equivalently  $\wh{\fAA}$ is an {\em $\wh{\LL}$-expansion} (or simply {\em expansion}) of $\fAA$,  when $\fAA$ and $\wh{\fAA}$ have the same underlying set and the interpretation in $\fAA$ of each constant, function and relation symbol of  $\LL$ coincides with the interpretation of that symbol in $\wh{\fAA}$. We denote the $\LL$-reduct of an $\wh{\LL}$-structure $\wh{\fAA}$ as ${\wh{\fAA}}\upharpoonright {\LL}$. 

\subsubsection*{Functions, Tuples and Formulas}

For sets $X$ and $Y$, we write $Y^X$ to denote the set of all functions from $X$ to $Y$.  Given $m \in \omega$, $[X]^m$ denotes the set of all $m$-element subsets of $X$.  Given an infinite set $X$, $A \subset_\omega X$ denotes that $A$ is a finite subset of $X$.  

\begin{notation}
Given a function $f: X \raw Y$ and $X_0 \subseteq X$, we write $f'' \, X_0$ for the image of $X_0$ under $f$, i.e., 
\[
    f'' X_0 : = \{ y \in Y : f(x)=y \text{ \ for some \ } x\in X_0\}.
\]
\end{notation}

Given a set $A$, a \textit{tuple from $A$} (or {\em sequence from $A$}) is an element $\ov{a} := (a_i)_{i\in \alpha}$ of $A^\alpha$ for some ordinal $\alpha$ (not necessarily finite), and we may write $\ov{a} \in A$ if no confusion arises. When $\ov{a} \in A^\alpha$, we call $\alpha$ the {\em length} of $\ov{a}$, denoted $|\ov{a}|$ (where there is no confusion with the domain of an $\LL$-structure.) An {\em $\alpha$-tuple} is a tuple of length $\alpha$. We define the {\em range} of $\ov{a}$ to be $\ran(\ov{a}):=\{a_i : i \in |\ov{a}| \}$.  Given sequences $s$ and $t$, the concatenation of $s$ with $t$ is denoted by $s^\smallfrown t$.  Given $\alpha \in |s|$, $s \uphp \alpha$ denotes the restriction of the sequence $s$ to $\alpha$ (the restriction of the sequence to the first $\alpha$ coordinates).

\begin{notation}\label{parameters}
For a first-order formula $\varphi$ and tuple of variables $\ov{x}$, the notation $\varphi(\ov{x})$ indicates that the free variables of $\varphi$ are elements of $\ran(\ov{x})$. For a formula $\psi$ and tuples of variables $\ov{x}$ and $\ov{y}$, the notation $\psi(\ov{x};\ov{y})$ indicates that the free variables of $\psi$ are among the elements of $\ran(\ov{x}) \cup \ran(\ov{y})$ and that $\ran(\ov{x})$ and $\ran(\ov{y})$ are disjoint.  Assuming that $\ran(\ov{y})$ and $\ran(\ov{a})$ are disjoint and $|\ov{y}|=|\ov{a}|$, by the notation $\psi(\ov{x};\ov{y}/\ov{a})$ we mean the formula obtained by replacing each occurrence of $y_i$ in $\psi(\ov{x};\ov{y})$ with $a_i$, for all $i \in |\ov{a}|$, what we may also write as $\psi(\ov{x};\ov{a})$.
If we work over an $\LL$-structure $\fAA$ and use $\varphi(\ov{a})$ or $\psi(\ov{x}; \ov{a})$ where $\ov{a}$ is a tuple from $A$, the elements of $\ran(\ov{a})$ are called {\em parameters} and $\varphi(\ov{a})$, $\psi(\ov{x}; \ov{a})$ are referred to as $\LL$-formulas with {\em parameters from $A$}. We write ``$\fAA \vDash \varphi(\ov{a})$'' to indicate that $\ov{a}$ satisfies $\varphi(\ov{x})$ in $\fAA$. This can be made formal by expanding the signature to include new constant symbols for the parameters, as explained in Ch.~5 of \cite{C-K}.

\end{notation}

A \textbf{literal} is either an atomic formula or the negation of an atomic formula, see \cite{ho93}.

\subsubsection*{Types and Saturation}

Types are consistent sets of first-order formulas all in the same number of variables. We will assume that all types are closed under (first-order) logical consequence.  
Types are usually treated as collections of formulas whose free variables come from some fixed finite set of variables. As noted on p. xxxiii of \cite{Sh-c}, it is sometimes useful to consider types in infinitely many variables, as we will see in Lemma \ref{technical}.

\begin{definition}\label{shdef} 
Fix $m \in \omega$. Let $\fAA$ be an $\LL$-structure, $B\subseteq A$, and $\ov{x}$ an $m$-tuple of variables. Suppose $p := p(\ov{x})$ is a set of $\LL$-formulas with parameters from $B$ such that every element of $p$ is of the form $\psi := \psi(\ov{x}; \ov{b})$ for some tuple $\ov{b}$ from $B$. Then $p$ is an {\bf $\bm{m}$-type over $\bm{B}$ in free variables $\bm{\ov{x}}$, with respect to $\bm{\fAA}$}, (or simply an {\bf $\bm{m}$-type}) if $p$ is finitely satisfiable in $\fAA$, i.e., for every finite $q \subseteq p$, we have $\fAA \vDash (\exists \ov{x}) \bigwedge_{\varphi \in q} \varphi(\ov{x})$.

We extend this definition to tuples of variables of infinite length by writing $p := p(\ov{x})$ with $\ov{x} := (x_i)_{i \in \alpha}$ for some  arbitrary ordinal $\alpha$ when $p$ is a finitely satisfiable set of $\LL$-formulas all of the form $\psi := \psi(x_{i_0},\ldots,x_{i_{n-1}}; \ov{b})$ for some finite sequence $x_{i_0},\ldots,x_{i_{n-1}}$ of variables
included in $\ran(\bar{x})$ and $\ov{b}$ from $B$. The sequence $x_{i_0},\ldots,x_{i_{n-1}}$ is allowed to change with $\psi$. In this case, we simply call $p$ a {\em type over $B$ in free variables $\ov{x}$.}

A maximal type over $\emptyset$, that is a maximal consistent set of formulas in some tuple of variables $\ov{x}$, 
is said to be {\em complete}. When we wish to emphasize that a type under consideration is not necessarily complete, we may call it {\em a partial type}.
\end{definition}

\begin{remark}\label{formality}
In many cases (such as in compactness arguments) it is useful to consider a type with variables from a set that is not an ordinal, in which case the set of variables is assumed to be re-indexed in some way as a sequence, possibly using AC, where the specific indexing is not important.  For example, given an ordinal $\alpha$, a sequence of variables $(x_i)_{i \in \alpha}$ and a subset $B \subseteq \alpha$, we may write $(x_i)_{i \in B}$.
\end{remark}

\begin{definition}\label{domp}
Let $\fAA$ be an $\LL$-structure and $p:= p(\ov{x})$ a type over $A$ in free variables $\ov{x}$, where $\ov{x}$ is a tuple of variables of arbitrary length. The {\bf domain of $p$}, denoted $\bm{\Domp(p)}$, is the set
\[
   \Domp(p) := \bigcap \{B \subseteq A : p \text{~is a type over $B$ in free variables $\ov{x}$}\},   
\]
i.e., $\Domp(p)$ is the smallest parameter set from $A$ over which $p$ is a type in free variables $\ov{x}$.
\end{definition}

Recall Notation \ref{parameters} describing formulas with parameters from some set.

\begin{definition} \label{realized}
Let $\fAA$ be an $\LL$-structure and $\ov{x} : = (x_i)_{i \in \alpha}$ a tuple of variables. A type $p(\ov{x})$ over $A$ in free variables $\ov{x}$ is {\bf realized in $\bm{\fAA}$} if there exists a sequence of parameters $\ov{a} :=(a_i)_{i \in \alpha}$ from $A$ such that $\fAA \vDash \psi(\ov{a})$ for all $\psi(\ov{x}) \in p(\ov{x})$.  We may say that $\ov{a}$ realizes $p(\ov{x})$ in $\fAA$.
\end{definition}

In particular, a 1-type $p(x)$ over $A$ in free variable $x$ is realized in $\fAA$ if there exists $a \in A$ such that $\fAA \vDash \varphi(a)$ for all $\varphi(x) \in p(x)$. Structures realizing 1-types over large domains are of particular interest.

\begin{definition} \label{kappasat}
Let $\kappa$ be an infinite cardinal.  A structure $\fAA$ is {\bf $\bm{\kappa}$-saturated (for 1-types)} if, whenever $p := p(x)$ is a 1-type over $A$ such that $|\Domp(p)|<\kappa$, the type $p$ is realized in $\fAA$.
\end{definition}

\begin{remark} For any $m\in \omega$, there is an obvious notion of {\em $\kappa$-saturation for $m$-types} parallel to Definition \ref{kappasat}. A structure is $\kappa$-saturated for 1-types if and only if it is $\kappa$-saturated for $m$-types for all $m \in \omega$, see Proposition \ref{ckprop}.
\end{remark}

It is well-known that among $\kappa$-saturated structures, elementary equivalence is sufficient to determine a structure of a given cardinality up to isomorphism.  
\begin{theorem}[Theorem 5.1.17 of \cite{C-K}]\label{uniquemodel} 
Let $\kappa$ be an infinite cardinal and let $\fAA$, $\fBB$ be $\LL$-structures. If  $\fAA \equiv \fBB$, $|A|=|B|=\kappa$, and both $\fAA$ and $\fBB$ are $\kappa$-saturated, then $\fAA \cong \fBB$.
\end{theorem}
Theorem \ref{uniquemodel} together with Corollary \ref{minimal_saturation} guarantee that,  under CH, there is a unique $\aleph_1$-dense linear order of size $\aleph_1$, as used in Section \ref{eta1}.

The next lemma plays a key role in the proofs of Theorems \ref{fabulous} and \ref{fabulous_general}. As its proof is left to the reader in the original reference \cite{Sh-c}, an exposition of the proof is provided in Section \ref{Appendix}, an appendix to this paper.

\begin{lemma}[Lemma 1.12 in Ch.~I of \cite{Sh-c}]\label{technical} Let $\lambda$ be an infinite cardinal and $\fMM$ a $\lambda$-saturated structure.  Suppose $p$ is a type over $M$ in free variables $\ov{x}$ with respect to $\fMM$  
such that $|\ov{x}| \leq \lambda$
and such that $|\Domp(p)|< \lambda$. Then $p$ is realized in $\fMM$.
\end{lemma}

\subsection{Ultrafilters and ultraproducts}

In the following, we let $I$ be an infinite set.

\begin{definition}  Given a set $I$,  a subset $\DD \subseteq \mathcal{P}(I)$ is a \emph{filter (over $I$)} if
	\begin{itemize}
	\item $I \in \DD$
	\item for any $X, Y \in \DD$, $X \cap Y \in \DD$
	\item for any $X \in \DD$ and $X \subseteq Y \subseteq I$, $Y \in \DD$.
	\end{itemize}

The filter $\DD$ is \emph{proper} if $\DD \neq \mathcal{P}(I)$.

The filter $\DD$ is \emph{principal} if there exists $X \in \DD$ such that for all $Y \in \DD$, $X \subseteq Y$.  In this case we say that $\DD$ is generated by $X$.

The filter $\DD$ is an \emph{ultrafilter} if for any set $X \in \mathcal{P}(I)$,  $X \in \DD$ if and only if $I \setminus X \notin \DD$.
\end{definition}

\begin{example} Given a set $I$, the Fr\'{e}chet filter $\DD := \{X \in \mathcal{P}(I) : I \setminus X \textrm{~is finite} \}$ is a filter.
\end{example}

\begin{remark}   It is well-known that AC implies that any proper filter over a set $I$ can be extended to an ultrafilter over $I$, hence providing a rich variety of ultrafilters.  For example, the Fr\'{e}chet filter can be extended to an ultrafilter.  In particular, the  Fr\'{e}chet filter over $I=\omega$ can be extended to an ultrafilter which must be nonprincipal, and every nonprincipal ultrafilter extends the Fr\'echet filter.
\end{remark}

Fix a sequence of $\LL$-structures $(\fMM_i)_{i \in I}$.  By $\prod_{i \in I} \fMM_i$ we denote the Cartesian product 
$\prod_{i \in I} M_i$ endowed with the coordinatewise interpretation of the symbols from $\LL$.
We frequently denote an element $a \in \prod_{i \in I} M_i$ as $\bigl (a[i] \bigr)_{i \in I}$ where $a[i] \in M_i$ for all $i \in I$.

Given an ultrafilter $\DD$ over $I$, we define the usual equivalence relation $\approx_\DD$ on elements $a, b \in \prod_{i \in I} M_i$, 
namely, 
$$
a \approx_\DD b  \textrm{ if and only if \quad } \{i : a[i]=b[i] \} \in \DD.
$$
If it is clear which $\DD$ we are working with, we may write $\approx$ in place of $\approx_\DD$.

Elements in the ultraproduct $\fMM^*:=\prod_{i \in I} \fMM_i / \DD$ will be written as $a \in \prod_{i \in I} M_i$ and understood up to $\approx$-equivalence.

Given $m \in \omega$ and a tuple $\ov{a} = (a_0,\ldots,a_{m-1})\in \fMM^*$, we denote
\[
 \ov{a}[i] := (a_0[i],\ldots,a_{m-1}[i]).
 \] 
Given a subset $U \subseteq \fMM^*$ and $i \in I$, define 
$U[i]:=\{u[i] : u \in U\}$. A fundamental theorem of ultraproducts is due to {\los} in \cite{Los} and using the notation just introduced, it states:

\begin{theorem}[\los's Theorem]\label{los}  For any $\ov{a} \in \fMM^\ast$ and
and a first-order $\LL$-formula  $\varphi(\ov{x})$ such that $|\ov{x}| = |\ov{a}|$, we have
\[
\fMM^* \vDash \varphi(\ov{a}) \textrm{\quad if and only if \quad } \bigl\{i : \fMM_i \vDash \varphi\bigl(\ov{a}[i]\bigr)\bigr\} \in \DD.
\]
\end{theorem}

Note that in the case that $\LL$ contains only relation symbols, we can easily project substructures of the ultraproduct down to any coordinate $t \in I$:
\begin{notation}\label{project}
Let $\LL$ be a relational signature and let $\DD$ be an ultrafilter over a set $I$.  Fix a sequence of $\LL$-structures $(\fMM_i)_{i \in I}$ and let $\fMM^*:=\prod_{i \in I} \fMM_i / \DD$.
Let $\AAA$ be a substructure of $\fMM^*$. For $t \in I$, write $\bm{\AAA[t]}$ for the substructure of $\MM_t$ that has the underlying set $A[t]:=\{a[t] : a \in A\}$.
\end{notation}
\noindent Note that $\AAA[t]$ defined above does not necessarily have as universe the whole of $M_t$, that is, it is possible that $A[t]$ is a proper subset of $M_t$.

\subsubsection*{Special properties of ultrafilters and saturation}
In this section, we define properties of ultrafilters that lead to saturation of ultraproducts.  Many of these tools were used to study Keisler's order, which is a robust classification scheme for first-order theories, see \cite{Ke67}.  More recent work on Keisler's order can be found in \cite{memSh}, \cite{AckKark} and \cite{Ulrich}.  This section closely follows the exposition in Chapter 6 of \cite{C-K}.   

\begin{definition} An ultrafilter $\DD$ is \emph{countably incomplete} if $\DD$ is not closed under countable intersections. Equivalently, there exists a sequence 
$(X_n)_{n\in\omega}$ of sets in $\DD$ whose intersection is empty. We say that $(X_n)_{n\in\omega}$ {\em witnesses} that $\DD$ is countably incomplete.
\end{definition}

\begin{remark}\label{nonp_ctbl}  Any nonprincipal ultrafilter $\DD$ over $I= \omega$ is countably incomplete.  To see this, define $X_n := \{i \in \omega : i > n\}$ for all $n \in \omega$.  By assumption, $X_n \in \DD$ (otherwise its complement $n$ is in $\DD$, thus making $\DD$ principal).  Since $\bigcap_{n \in \omega} X_n = \emptyset \notin \DD$, the sets $(X_n)_{n \in \omega}$ witness that $\DD$ is countably incomplete.
\end{remark}

In \cite{C-K},  $\kappa$-good ultrafilters are defined, for a cardinal $\kappa$.   Since the definition is quite technical, let us merely note the following:

\begin{theorem}[Theorem 6.1.4 of \cite{C-K}]\label{general_exists} For any set $I$ of cardinality $\kappa$ there exists a $\kappa^+$-good countably incomplete ultrafilter $\DD$ over $I$.
\end{theorem}

\begin{theorem}[Theorem 6.1.8 \cite{C-K}]\label{maximal_saturation} Let $\kappa$ be an infinite cardinal and let $\DD$ be a countably incomplete $\kappa$-good ultrafilter over a set $I$.  Suppose $|\LL|<\kappa$.  Then for any family of $\LL$-structures $(\fAA_i)_{i \in I}$, $\prod_{i \in I} \fAA_i / \DD$ is $\kappa$-saturated.
\end{theorem}

\begin{corollary}\label{minimal_saturation} Suppose $\LL$ is countable and $\DD$ is a nonprincipal ultrafilter over $\omega$.  
Then for any family $(\fAA_i)_{i \in \omega}$ of $\LL$-structures, $\prod_{i \in \omega} \fAA_i / \DD$ is $\aleph_1$-saturated.
\end{corollary}

\begin{proof}  Claim 2.1 in Ch.~VI of \cite{Sh-c} states that every filter is $\aleph_1$-good,  thus, any ultrafilter $\DD$ over $I=\omega$ is $\aleph_1$-good, and if it is nonprincipal, it is also countably incomplete, by Remark \ref{nonp_ctbl}.  The rest follows by Theorem \ref{maximal_saturation} and setting $\kappa:=\aleph_1$.
$\eop_{\ref{minimal_saturation}}$
\end{proof}

\subsubsection*{Pseudofinite structures and cardinality}

We shall mostly be interested in the ultraproducts of finite structures, so we shall now pass to that case. Structures elementarily equivalent to substructures of such ultraproducts are called {\em pseudofinite}. In what follows, we shall assume that $\fMM^*:=\prod_{i \in I} \MM_i / \DD$ for some ultrafilter $\DD$ over some infinite set $I$ and some finite structures $(\MM_i)_{i\in I}$.

If $I = \omega$, clearly $|\fMM^*| \leq \mathfrak c$ as the size of the (non-reduced) product of finite sets is bounded above by $\mathfrak c$.
It is not at all an easy question to determine the size of the ultraproducts of finite sets (see \cite{Sh:1026} for a history and the definitive theorem in the subject, Theorem 1.2.).  However, under certain frequently encountered assumptions, the size of the ultraproduct $\fMM^*$ is exactly the continuum $\mathfrak c$:

\begin{theorem}[\cite{FMScott}]\label{sizec}
Let $\LL$ be a signature, $\DD$ a nonprincipal ultrafilter over $\omega$ and $(\MM_i)_{i\in \omega}$ a sequence of finite $\LL$-structures such that the ultraproduct $\fMM^*:= \prod_{i \in \omega} \MM_i / \DD$ is infinite. Then $\fMM^*$ has cardinality $2^{\aleph_0}$. 
\end{theorem}

Theorem 2.13 in Chapter VI of \cite{Sh-c}
states that for any $|I|^+$-good ultrafilter over $I$,  for any integers $(n_i)_{i \in I}$, if $|\prod_{i \in I} n_i / \DD| \geq \aleph_0$, then $|\prod_{i \in I} n_i / \DD| \geq 2^{|I|}$.  On the other hand, a product $\prod_{i \in I} n_i$ of finite sets has cardinality at most $2^{|I|}$, and the cardinality of the reduced product is bounded by that size.  
This gives us the general case:

\begin{fact}\label{sizegc}
Let $\LL$ be a signature, $\DD$ a $\kappa^+$-good countably incomplete ultrafilter over $\kappa$ and $(\MM_i)_{i\in \kappa}$ a sequence of finite $\LL$-structures such that the ultraproduct $\fMM^* := \prod_{i \in \kappa} \MM_i / \DD$ is infinite.  Then  $\fMM^*$ has cardinality $2^{\kappa}$. 
\end{fact}

\subsubsection*{Finite substructures of ultraproducts}

To define ``internal'' colorings in the next section, we need to first develop notation for defining finite substructures of an ultraproduct.

For the rest of this subsection, fix a finite relational signature $\LL$, a nonprincipal ultrafilter $\DD$ over $\omega$, and a sequence $(\MM_i)_{i \in \omega}$ of finite (nonempty) $\LL$-structures. For each $i \in \omega$, write $M_i$ for the  underlying set of $\MM_i$. Write $\fMM^* := \prod_{i \in \omega} \MM_i \, / \,  \DD$. 

Let $<_\omega$ denote the usual order on $\omega$. We may assume without loss of generality that each $\MM_i$ has some positive integer as its underlying set $M_i$. Let $<_i$ denote the restriction of $<_\omega$ to $M_i$. The ultraproduct $\fMM^*$ has a natural linear order $\prec$  obtained from the finite linear orders $<_i$, given by: for $a, b \in \fMM^*$,
\[
a \prec b \textrm{ \quad if and only if \quad } \bigl\{i: a[i] <_i b[i] \bigr\} \in \DD.
\]

Let $\fNN$ be an $\LL$-structure, $\AAA$ a finite substructure of $\fNN$, and $n := |A|$. We say that a tuple $\ov{a}$ from $N$ with $|\ov{a}| = n$ is an \emph{enumeration of $\AAA$ in $\fNN$} if $\ran(\ov{a})=A$. 

Suppose $<_N$ is some linear order on $N$ (not necessarily interpreting a relation symbol in $\LL$). An enumeration $\ov{a}$ of $\AAA$ in $\fNN$ is the \emph{increasing enumeration of $\AAA$ in $\fNN$ with respect to $<_N$} if, 
for all $i,j \in n$, we have $a_i <_N a_j$ precisely when $i\in j$. When $\fNN = \fMM^*$ or when $\fNN = \MM_i$ for some $i \in \omega$, we assume that $<_N$ is $\prec$ or $<_i$, respectively, and suppress mention of the order on $N$.

\begin{definition}\label{thetaA}
Let $\fNN$ and $<_N$ be as above. Given a finite substructure $\AAA$ of $\fNN$, write $n := |A|$ and let $\ov{a}$ be the increasing enumeration of $\AAA$ in $\fNN$ with respect to $<_N$. We shall define {\em the quantifier-free type of $\ov{a}$ in $\fNN$}: it is the unique up to logical equivalence quantifier-free $\LL$-formula $\theta_\AAA(\ov{x})$, where $|\ov{x}| = n$, satisfying:
\begin{itemize}
    \item 
    $\fNN \vDash \theta_\AAA(\ov{a})$, and
    \item\label{theta2}
    for any $\LL$-structure $\fNN'$, any substructure $\BB$ of $\fNN'$ such that $|B| = n$, and any enumeration $\ov{b}$ of $B$, if $\fNN' \vDash \theta_\AAA(\ov{b})$ then the map sending $a_i$ to $b_i$ for each $i\in n$ is an $\LL$-isomorphism.
\end{itemize}
\end{definition}

\noindent Such a (first-order) formula $\theta_\AAA(\ov{x})$ exists because $\LL$ is finite and relational. 

\begin{remark}\label{strongerfortheta}
In fact, a stronger property holds of $\theta_\AAA(\ov{x})$ than required by the definition above: for any $\LL$-structure $\fNN'$ and finite substructure $\BB$ of $\fNN'$, we have that $\BB \cong \AAA$ if and only if there is some enumeration $\ov{b}$ of $\BB$ in $\fNN'$ such that $\fNN' \vDash \theta_{\AAA}(\ov{b})$.
\end{remark}

For the following, refer to Definition \ref{project} for the definition of $\AAA[i]$:

\begin{observation}\label{whendefined}
For any finite substructure $\AAA$ of $\fMM^*$, we have
\[
    \{i : \AAA[i] \cong \AAA\} \in \DD.  
\]
\end{observation}

\begin{proof} 
Let $\ov{a}$ be the increasing enumeration of $\AAA$ in $\fMM^*$. Then $\fMM^* \vDash \theta_\AAA(\ov{a})$, and so by  \los's Theorem, we have $\{i: \MM_i \vDash \theta_\AAA(\ov{a}[i]) \} \in \DD$. Since, for each $i$, the tuple $\ov{a}[i]$ is an enumeration of $\AAA[i]$ in $\MM_i$, by Remark \ref{strongerfortheta} we have that 
$ \{i : \MM_i \vDash \theta_\AAA(\ov{a}[i])\} \subseteq \{i : \AAA[i] \cong \AAA\}$. Hence $\{i : \AAA[i] \cong \AAA\} \in \DD$.
$\eop_{\ref{whendefined}}$
\end{proof}

\section{Partition properties and internal colorings}\label{intercolo}

In this section, we introduce some notions from structural Ramsey theory, as well as a special type of coloring on an ultraproduct called an {\em internal coloring}.

\subsection{Small and big Ramsey degrees}\label{Ramseyprops}

For the rest of this subsection, fix a signature $\LL$. Recall Notation \ref{structureRamsey} and Definition \ref{smallRamseydeg} from the Introduction, where we have defined Ramsey objects and Ramsey degrees. Let us take a few sentences to give a more detailed history of these notions.

We can understand Finite Ramsey Theorem \cite{Ramseyth},
as stating that any finite set (in the empty signature) has the Ramsey property in the class of finite sets.
Abramson and Harrington in \cite{AbramsonHarrington} and independently Ne\v{s}et\v{r}il and R\"odl in \cite{NesetrilRodl} showed that any finite ordered graph is a Ramsey object in the class of finite ordered graphs. However,
only complete and empty graphs are Ramsey objects in the class of all finite graphs (see \cite{NesetrilRodlcomplete} for a general proof and history of special cases). 
The notion of small Ramsey degree was introduced by Fouch\'e in \cite{FOUCHE1997309}
as a parameter measuring how far a finite structure is from having the Ramsey property in a given class.
Ne\v{s}et\v{r}il and R\"odl showed that finite graphs satisfy the Ramsey property for colorings of vertices in \cite{NesetrilRodlvertex}, but not every graph is a Ramsey object in the class of all graphs. However, it follows from the fact that finite ordered graphs are a Ramsey class that every graph has
finite small Ramsey degree in the class of all graphs.

Going to the infinite structures, as mentioned in the Introduction
the example of Sierpi\'{n}ski in \cite{MR1556708}
shows that for $\AAA$ the 2-element linear order, the partition relation $\mathbb{Q} \lraw (\mathbb{Q})^\AAA_{2,1}$ fails. Galvin (unpublished)
later showed that for any $k \in \omega$, the relation $\mathbb{Q} \lraw (\mathbb{Q})^\AAA_{k,2}$ holds. This and similar phenomena were the motivation for the formulation in \cite{KPT}
of the notion of the big Ramsey degree of an infinite structure.

As for the relation between the small and the big Ramsey degrees, it is not difficult to see using a compactness argument (see e.g. the proof of Proposition 3 in \cite{NguyenVanThe}), that the following holds:

\begin{observation}\label{big-implies-small}
Let $\fNN$ be an infinite, locally finite $\LL$-structure and $\AAA$ a finite $\LL$-structure. Suppose $\AAA$ has finite big Ramsey degree in $\fNN$. Then $\AAA$ has finite small Ramsey degree in $\age(\fNN)$, and $t(\AAA, \age(\fNN)) \le T(\AAA, \fNN)$.
\end{observation}

In known examples, typically the big Ramsey degrees are strictly larger (possibly infinite) than small Ramsey degrees.  Corollary \ref{CHcorollary}  of this paper shows that a version of the reverse inequality holds in the uncountable context of ultraproducts, with a restricted notion of coloring and under the assumption of CH.

\subsection{Internal colorings}

For the rest of this subsection, fix a finite relational signature $\LL$, 
a nonprincipal ultrafilter $\DD$ over $\omega$, and a sequence $(\MM_t)_{t \in \omega}$ of finite $\LL$-structures. For each $i \in \omega$, write $M_i$ for the  underlying set of $\MM_i$. Write $\fMM^* := \prod_{i \in \omega} \MM_i \, / \,  \DD$. 

Our main result, Theorem \ref{fabulous}, concerns partition relations for finite substructures of $\fMM^*$ under a restricted class of colorings, which we call {\em internal} colorings, that are controlled by the ultrafilter $\DD$. In order to define them, we need some notation.

\begin{notation}\label{ZeeCeeJay}
Fix $k \in \omega$ and a finite substructure $\AAA$ of $\fMM^*$. Let $Z \in \DD$ and let $\sC : = \{c_i : i \in Z\}$ be a collection of $k$-colorings of copies of $\AAA$ in $\MM_i$, for each $i \in Z$. For each $j \in k$, write
\[
   Z_\AAA^{\sC,j} : = \{i \in Z : \AAA[i] \cong \AAA \text{ and } c_i(\AAA[i])= j\}.
\]
\end{notation}

\begin{definition}\label{internalcolorcopies} 
Fix $k \in \omega$ and a finite substructure $\AAA$ of $\fMM^*$. Suppose $c^*$ is a $k$-coloring of the copies of $\AAA$ in $\fMM^*$.  We say that $c^*$ is {\bf internal} if there exist $Z \in \DD$ and a collection $\sC : = \{c_i : i \in Z\}$, where each $c_i$ is a $k$-coloring of the copies of $\AAA$ in $\MM_i$, such that for any copy $\AAA'$ of $\AAA$ in $\fMM^*$ and any $j \in k$, we have $c^*(\AAA') = j$ if and only if $Z_{\AAA'}^{\sC,j} \in \DD$.

A $k$-coloring of the copies of $\AAA$ in $\fMM^*$ that is not internal will be called {\bf external}.
\end{definition}

The following proposition demonstrates that internal colorings are easy to construct; any collection of colorings $\{c_i : i \in Z \}$ as in Definition \ref{internalcolorcopies} gives rise to an internal coloring of $\fMM^*$.

\begin{proposition}\label{hope} 
Fix $k \in \omega$ and a finite substructure $\AAA$ of $\fMM^*$. Let $Z \in \DD$ and let $\sC : = \{c_i : i \in Z\}$, where $c_i$ is a $k$-coloring of the copies of $\AAA$ in $\MM_i$ for each $i \in Z$.  Consider the map $c^* : \binom{\fMM^*}{\AAA} \raw k$ given by: For any copy $\AAA'$ of $\AAA$ in $\fMM^*$ and any $j \in k$, let $c^*(\AAA') = j$ if and only if $Z_{\AAA'}^{\sC,j} \in \DD$. Then $c^*$ is well defined, and hence $c^*$ is an internal $k$-coloring of the copies of $\AAA$ in $\fMM^*$. 
\end{proposition}
\begin{proof} Let $\AAA' \in \binom{\fMM^*}{\AAA}$ and write $Z_{\AAA'} : = \{i : \AAA'[i] \cong \AAA'\} \cap Z$. By Proposition \ref{whendefined} and since $Z \in \DD$, we have $Z_{\AAA'} \in \DD$. Notice that the set \{$Z_{\AAA'}^{\sC,j}$ : $j \in k \}$  is a finite partition of $Z_{\AAA'}$. Therefore, by a standard property of ultrafilters, there exists a unique $j \in k$ such that $Z_{\AAA'}^{\sC,j} \in \DD$. Thus $c^*$ is well defined. 
$\eop_{\ref{hope}}$
\end{proof}

When working within an ambient ultraproduct, a natural variation on the partition properties described in Subsection \ref{Ramseyprops} is to restrict to colorings that are internal. 
We use ``decorated'' arrow notation for such restricted partition relations.

\begin{notation}\label{arrowint}
Let $\AAA$ be a finite substructure of $\fMM^*$ and $\fNN$ an $\LL$-structure.
For $k, \ell \in \omega$, write
\[
    \fMM^* \lrawi \bigl( \fNN \bigr)^\AAA_{k, \ell}
\]
to mean that for any internal $k$-coloring $c$ of the copies of $\AAA$ in $\fMM^*$ there exists a copy $\fNN'$ of $\fNN$ in $\fMM^*$ such that $\binom{\fNN'}{\AAA}$ is $\ell$-chromatic by $c$.
\end{notation}

In order to state our main result, Theorem \ref{fabulous}, we will need the counterpart of the notion of big Ramsey degree for internal colorings within ultraproducts.  

\begin{definition}\label{internalbRd}
Let $\AAA$ be a finite substructure of $\fMM^*$. We say that $\AAA$ has {\bf finite internal big Ramsey degree in $\fMM^*$} if there exists $\ell \in \omega$ such that for each $k \in \omega$, the partition relation 
\[ 
    \fMM^* \lrawi \bigl(\fMM^* \bigr)^\AAA_{k,\ell} \, 
\]
holds. The smallest such $\ell \in \omega$ (if it exists) is called the {\bf internal big Ramsey degree of $\bm{\AAA}$ in $\bm{\fMM^*}$} and denoted $\bm{ T_\intern(\AAA,\fMM^*)}$. If no such $T_\intern(\AAA, \fMM^*)$ exists, we say that $\AAA$ has \emph{infinite internal big Ramsey degree in $\fMM^*$}. 

\end{definition}

\section{Big Ramsey degrees for internal colorings: $\omega$ case}\label{omegacase}

In this section we prove the main theorem of this paper, Theorem \ref{fabulous}, and deduce Corollary \ref{CHcorollary} on internal big Ramsey degrees for ultraproducts under CH. 

Fix a finite relational signature $\LL$, 
a nonprincipal ultrafilter $\DD$ over $\omega$, and a sequence $(\MM_i)_{i \in \omega}$ of finite $\LL$-structures. For each $i \in \omega$, write $M_i$ for the  underlying set of $\MM_i$. Write $\fMM^* := \prod_{i \in \omega} \MM_i \, / \,  \DD$.

Our main theorem concerns ultraproducts of sequences of finite $\LL$-structures that satisfy a mild coherence condition called {\em $\DD$-trending}, which we now describe. 

\begin{definition}\label{trending} 
We say that a sequence $(\MM_i)_{i \in \omega}$ of finite (nonempty) $\LL$-structures is {\bf $\mathbf\DD$-trending} if 
	\begin{enumerate}
	\item\label{incr} 
	$|M_i| \leq |M_j|$ for all $i \in j \in \omega$,

	\item\label{infty} 
	$\lim_{i \raw \infty} |M_i| = \infty$, and 

	\item\label{persist} $\{j : \MM_i \hookrightarrow \MM_j \} \in \DD$ for each $i\in\omega$. 
\end{enumerate}
\end{definition}

We now make some observations relating $\DD$-trending sequences and ages of  $\LL$-structures. 
\begin{observation}\label{agecontain} 
Suppose that $(\MM_i)_{i\in \omega}$ is $\DD$-trending. Then 
\[
    \displaystyle\bigcup\{\age(\MM_i) : i \in \omega \} = \age(\fMM^*).
\]
\end{observation}

\begin{proof}
Let $\AAA \in \age(\MM_i)$ for some $i \in \omega$. Then $\MM_i \vDash \exists \ov{x}\, \theta_\AAA(\ov{x})$, where $\theta_\AAA(\ov{x})$ is as in Definition \ref{thetaA}. Hence, $\MM_j \vDash \exists \ov{x}\, \theta_\AAA(\ov{x})$ for all $j$ such that $\MM_i$ embeds into $\MM_j$. By item \ref{persist} of Definition \ref{trending} and \los's Theorem, it follows that $\fMM^*\vDash \exists \ov{x}\, \theta_\AAA(\ov{x})$. Hence  $\AAA\in \age(\fMM^*)$. The reverse inclusion is immediate by \los's Theorem, and does not require that $(\MM_i)_{i\in\omega}$ is $\DD$-trending.
$\eop_{\ref{agecontain}}$
\end{proof}

The next observation asserts that the age of any $\LL$-structure contains a sequence of elements that form a cofinal chain in the partial order of $\LL$-embeddings on the age. Recall that when $\LL$ is finite and relational, there are only countably many isomorphism types among the elements of the age of any $\LL$-structure.

\begin{observation}\label{canbedone} 
Let $\fNN$ be an $\LL$-structure and suppose that $\sI : = (\AAA_r)_{r \in \omega}$ is an enumeration (possibly with repetitions) 
of elements of $\age(\fNN)$ such that $\sI$ contains a representative of each isomorphism class of structures in $\age(\fNN)$. Then there is a sequence $(\BB_i)_{i \in \omega}$ of finite $\LL$-structures such that 
\begin{enumerate}
    \item \label{subsequence}
    $\BB_i \in \sI$ for all $i \in \omega$;
    
    \item \label{chain}
    $\BB_{i-1} \hookrightarrow \BB_i$ for all $i \in \omega$, $i \ne 0$; and
    
    \item \label{cofinal}
    $\AAA_i \hookrightarrow\BB_j$ for all $i \in j \in \omega$.
\end{enumerate}
Further, if $\fNN$ is infinite, then $(\BB_i)_{i\in\omega}$ is $\DD$-trending.
\end{observation}

\begin{proof}
 Define $\BB_0 = \AAA_0$. Assume, for some $t \ne 0$, that $\BB_{i-1}$ exists as required. Since $\sI$ contains a representative of each isomorphism type in $\age(\fNN)$, it has the Joint Embedding Property. 
 Hence, there exists 
 $\BB_i \in \sI$ such that $\BB_{i-1} \hookrightarrow \BB_i$ and $\AAA_{i-1} \hookrightarrow \BB_t$. Then $\BB_i$ is as required.
 
 Note that the sequence $(\BB_i)_{i \in \omega}$ thus constructed satisfies item \eqref{persist} of Definition \ref{trending} for $\fNN$ of finite or infinite cardinality. This is because for any $j \in \omega$, $\{i : j \le i\} \in \DD$ since $\DD$ is nonprincipal, and by construction, $\{i : j \le i\} \subseteq \{i : \BB_j \hookrightarrow \BB_i \}$. 
 
 When $\fNN$ is infinite there are elements of $\age(\fNN)$, hence of $\sI$, of arbitrarily large finite cardinality. In that case, the sequence $(\BB_i)_{i \in \omega}$ as constructed must satisfy items \ref{incr} and \ref{infty} of Definition \ref{trending} as well. 
$\eop_{\ref{canbedone}}$
\end{proof}

\begin{proposition}\label{main_extracted} Suppose that
$(\MM_i)_{i\in\omega}$ is $\DD$-trending and let $\AAA$ be a finite substructure of $\fMM^*$. Further suppose that $k \in \omega$ and we are given an internal $k$-coloring $c^*$ of copies of $\AAA$ in $\fMM^*$.  

Let $\widehat{\LL}$ be the signature consisting of $\LL$ along with a new $n$-ary function symbol $f$ and new constant symbols $e_0,\ldots,e_{k-1}$.  Then there exist $\widehat{\LL}$-expansions $\widehat{\MM}_i$ of $\MM_i$, for $i \in \omega$, such that 
for any copy $\AAA'$ of $\AAA$ in $\fMM^*$, 
any enumeration $\ov{b}$ of the underlying set of $\AAA'$, and any $j\in k$, we have $f^{\widehat{\fMM^*}}(\ov{b}) = e_j^{\widehat{\fMM^*}}$ if and only if $c^*(\AAA') = j$, where $\widehat{\fMM^*}:=\prod_{i \in \omega} \widehat{\MM}_i / \DD$.
\end{proposition}

\begin{proof}
	Fix an internal $k$-coloring $c^*$ of copies of $\AAA$ in $\fMM^*$.  Then, by Definition \ref{internalcolorcopies} and recalling Notation \ref{ZeeCeeJay}, the coloring $c^*$ is determined by a set $Z \in \DD$ and collection $\sC := \{c_i : i \in Z\}$, 
	where each $c_i$ is a $k$-coloring of the copies of $\AAA$ in $\MM_i$, in the following manner:  $c^*$ assigns $j \in k$ to a copy $\AAA'$ of $\AAA$ in $\fMM^*$ if and only if $Z_{\AAA'}^{\sC, j} \in \DD$.

Since the sequence $(\MM_i)_{i\in \omega}$ is $\DD$-trending by hypothesis and $\DD$ is nonprincipal, so it contains the cofinite sets, we have $Z' : = \{i : |M_i| \ge k \} \in \DD$. Hence, by replacing $Z$ with $Z \cap Z'$, we may assume that $|M_i| \geq k$ for all $i \in Z$.

Write $n:=|A|$.  

For each $i\in\omega$, define an expansion of $\MM_i$ to an $\widehat{\LL}$-structure $\widehat{\MM}_i$ as follows: 

\begin{enumerate}

\item [(i)] 
the $e_0^{\widehat{\MM}_i}, \ldots, e_{k-1}^{\widehat{\MM}_i}$ are  distinct elements of $\widehat{\MM}_i$ if $i \in Z$, 
and arbitrary elements of $\widehat{\MM}_i$ if $i \notin Z$; and 

\item[(ii)]
for any $n$-tuple $\ov{s}$ consisting of 
elements from $M_i$,
letting $\mathcal{X}_{\ov{s}}$ denote the substructure of $\MM_i$ having underlying set $\ran(\ov{s})$, 
\[
f^{\widehat{\MM}_i}(\ov{s}) :=
\begin{cases}
e_j^{\widehat{\MM}_i} & \text{ if } \ i \in Z, \ \mathcal{X}_{\ov{s}} \cong \AAA \ \text{ and } \ c_i(\mathcal{X}_{\ov{s}})=j, \\ \\ 
e_0^{\widehat{\MM}_t} & \text{ otherwise}.
\end{cases}
\]
\end{enumerate}

\noindent Write $\widehat{\fMM^*}:=\prod_{i \in \omega} \widehat{\MM}_i / \DD$.  It is easy to see that ultraproducts commute with reducts and so $\widehat{\fMM^*} \uphp \LL 
= \fMM^*$
(e.g., Theorem 4.1.8 in \cite{C-K}.) 

The intuition is that for each $i\in \omega$, the constants $e_0^{\widehat{\MM}_i}, \ldots , e^{\widehat{\MM}_i}_{k-1} \in M_i$ code the color classes of $c_i$ and the function $f^{\widehat{\MM}_i}$ keeps track of the color assigned by $c_i$ to each copy of $\AAA$ in $\MM_i$.
We will show that $f^{\widehat{\fMM^*}}$ keeps track of the color assigned by $c^*$ to each copy of $\AAA$ in $\fMM^*$, in the following sense: for any copy $\AAA'$ of $\AAA$ in $\fMM^*$, 
any enumeration $\ov{b}$ of the underlying set of $\AAA'$, and any $j\in k$, we have $f^{\widehat{\fMM^*}}(\ov{b}) = e_j^{\widehat{\fMM^*}}$ if and only if $c^*(\AAA') = j$.

Note from the definition of $f^{\widehat{\MM}_t}$ that if $\ov{s}, \ov{s}'$ enumerate the same $n$-element subset of $M_i$, then $f^{\widehat{\MM}_i}(\ov{s})=f^{\widehat{\MM}_i}(\ov{s}')$.
Consequently, if $\ov{b}, \ov{b}'$ enumerate the same $n$-element subset of the underlying set of $\widehat{\fMM^*}$, then we have
\[
   \bigl \{i : \widehat{\MM}_i \vDash f\bigl (\ov{b}[i] \bigr)=f\bigl (\ov{b}'[i] \bigr ) \bigr \} = \omega \in \DD,
\]
and so by \los's Theorem, $\widehat{\fMM^*} \vDash f(\ov{b}) = f(\ov{b}')$.

Let $\AAA' \in \binom{\fMM^*}{\AAA}$, let $\ov{b}$ be an enumeration of the underlying set of $\AAA'$, and let $j := c^*(\AAA')$. 
We wish to show that 
\[
    Y^j := \{i : \widehat{\MM}_i \vDash f(\ov{b}[i]) = e_j \} \in \DD, 
\]
in order to conclude that $\widehat{\fMM^*} \vDash f(\ov{b}) = e_j$. 
Since $c^*$ is internal, 
we know that $Z_{\AAA'}^{\sC, j} \in \DD$.
For any $i \in Z_{\AAA'}^{\sC, j}$, we have that $i\in Z$ and $\AAA'[i] \cong \AAA'$ and $c_i(\AAA'[i]) = j$.  The substructure of $\MM_i$ with underlying set $\ov{b}[i]$ is $\AAA'[i]$, so by the definition of $f^{\widehat{\MM}_i}$ we have $f^{\widehat{\MM}_i}(\ov{b}[i]) = e_j^{\widehat{\MM}_i}$, and thus $\widehat{\MM}_i \vDash f(\ov{b}[i]) = e_j$.  This shows that $Z_{\AAA'}^{\sC, j} \subseteq Y^j$, and so $Y^j \in \DD$ as desired.
$\eop_{\ref{main_extracted}}$
\end{proof}

For the following result,
recall Notation \ref{arrowint} for partition relations restricted to internal colorings.

\begin{theorem}\label{fabulous}  
Let $\LL$ be a finite relational signature, $\DD$ a nonprincipal ultrafilter over $\omega$, and $(\MM_i)_{i\in\omega}$ a $\DD$-trending sequence of finite $\LL$-structures.  Define $\fMM^* := \prod_{i \in \omega} \MM_i / \DD$.  Let $\fNN$ be an $\LL$-structure of cardinality at most $\aleph_1$ such that
\[
    \K := \age(\fNN) \subseteq \bigcup \{\age(\MM_i) : i \in \omega \}.
\]
Fix $\AAA \in \K$ and suppose $\AAA$ has finite small Ramsey degree in $\K$, with $t := t(\AAA, \K)$. Then for any $k \in \omega$, the partition relation $ \fMM^* \lrawi \bigl( \fNN \bigr)^\AAA_{k, t}$ holds.
\end{theorem}

\begin{proof} 
Let $\LL$, $\DD$, $(\MM_i)_{i\in \omega}$, $\fMM^*$, $\fNN$, $\K$, $\AAA$ be as in the statement of the theorem. By $\theta_\AAA(\ov{x})$ we mean the quantifier-free formula defined in Definition \ref{thetaA}, where we set $\fNN$ to be $\fMM^*$. 
 Note that the conclusion of the theorem is non-vacuous, as Observation \ref{agecontain} tells us that $\binom{\fMM^*}{\AAA}$ is nonempty and Proposition \ref{hope} tells us that internal colorings of $\binom{\fMM^*}{\AAA}$ exist.
	
	By assumption, $t:=t(\AAA,\K) \in \omega$.  Let $k \in \omega$.  Our goal is to show that for any internal $k$-coloring $c$ of copies of $\AAA$ in $\fMM^*$, there exists a copy $\fNN'$ of $\fNN$ in $\fMM^*$ such that $\binom{\fNN'}{\AAA}$ is $t$-chromatic by $c$.
	
Fix an internal $k$-coloring $c^*$ of copies of $\AAA$ in $\fMM^*$.  Let $\widehat{\LL}$ be the signature consisting of $\LL$ along with a new $n$-ary function symbol $f$ and new constant symbols $e_0,\ldots,e_{k-1}$.  By Proposition \ref{main_extracted} there exist $\widehat{\LL}$-expansions $\widehat{\MM}_t$ of $\MM_i$ for all $i \in \omega$ such that if we denote $\widehat{\fMM^*}:=\prod_{i \in \omega} \widehat{\MM}_i / \DD$, then 
for any copy $\AAA'$ of $\AAA$ in $\fMM^*$, 
any enumeration $\ov{b}$ of the underlying set of $\AAA'$, and any $j\in k$, we have $f^{\widehat{\fMM^*}}(\ov{b}) = e_j^{\widehat{\fMM^*}}$ if and only if $c^*(\AAA') = j$.
Recall from the proof of Proposition \ref{main_extracted} that \\$\widehat{\fMM^*} \uphp \LL = \fMM^*$.
Since the interpretation of $f$ is not sensitive to the enumeration, we can speak of $f^{\widehat{\fMM^*}}$ as acting directly on ${\fMM^* \choose \AAA}$.
Thus, it suffices to show that  there is a substructure $\fNN'$ of $\fMM^*$ isomorphic to $\fNN$ and there exists
$S_0 \subseteq \{e_j^{\widehat{\fMM^*}} : j \in k\}$ such that $|S_0| = t$ and with the property:
\begin{eqnarray}\label{main_key}
\textrm{~for any~} \AAA' \in {\fNN' \choose \AAA}, \textrm{~there is an enumeration~} \ov{b} \textrm{~of~} \AAA' \textrm{~such that~} f^{\widehat{\fMM^*}}(\ov{b}) \in S_0.
\end{eqnarray}

By the definition of $t$, for any $k$, in particular for $k$ as above, for any desired $\BB \in \K$ there exists a $\CC \in \K$ such that $\CC \raw(\BB)^\AAA_{k,t}$; we will choose one such $\CC$ for each $\BB$ and denote it by $\CC_\BB$.  Since any element of $\K$ embeds in $\fMM^*$, by our assumptions and Observation \ref{agecontain}, we must have that $X_\BB :=\{i : \CC_\BB \hookrightarrow \MM_i\} \in \DD$ by \los's theorem.

Given any $\BB \in \K$ and $S \subseteq k$,
let $\varphi_{\BB,S}$ express that ``there is a copy of $\BB$ all of whose copies of $\AAA$ take on only colors from $S$ under $f$''.
Let $n:=|A|$.

$$\varphi_{\BB,S} := \exists \ov{x} (\theta_\BB(\ov{x}) \wedge \left( \bigwedge_{\text{$\ran(\ov{y}) \subseteq \ran(\ov{x})$,}\atop \text{$|\ov{y}|=n$}} [\theta_\AAA(\ov{y}) \raw \bigvee_{s \in S} f(\ov{y}) = e_{s}]\right) .$$

Now we define a formula $\varphi_\BB$ that expresses that ``there is a $t$-chromatic copy of $\BB$". Note that the disjuncts in $\varphi_\BB$ are not necessarily mutually exclusive:
$$\varphi_\BB :=  \bigvee_{S \subseteq k, |S|=t}  \varphi_{\BB,S}$$

For any $\BB \in \KK$, any ${\MM}_i$ which embeds a copy of $\CC_\BB$ (and is size at least $k$) 
will satisfy $\varphi_\BB$, by the existence of $t=t(\AAA,\K)$. So, in fact

$$\DD \ni X_{\BB} \subseteq \{i : \widehat{\MM}_i \vDash \varphi_\BB \}$$

\noindent and thus $\widehat{\fMM^*} \vDash \varphi_\BB$ by \los's theorem.

Thus, for any $\BB \in \KK$, $\widehat{\fMM^*}$ satisfies some disjunct of $\varphi_\BB$, where this disjunct selects a specific size-$t$ subset of $k$.  So,
for every $\BB \in \KK$, let $S(\BB):=$ one such $S$ ($S \subseteq k$ such that $|S| = t$) such that $\fMM^* \vDash \varphi_{\BB,S}$.  Thus,
\begin{eqnarray}\label{1main}
\fMM^* \vDash \varphi_{\BB,S(\BB)}.
\end{eqnarray}
Fix a sequence
$(\BB_i)_{i \in \omega}$ 
from $\KK$ satisfying the conditions in Observation \ref{canbedone}.
There are finitely many size-$t$ subsets $S \subseteq k$.  Thus, there is some subset of size $t$, $S_0 \subseteq k$, such that $G=\{i : S(B_i) = S_0 \}$ is an infinite subset of $\omega$.

The following are partial $\hat{\LL}$-types in $\aleph_1$-many variables.
First we define the atomic diagram of $\fNN$ in the variables $(x_\alpha)_{\alpha \in \omega_1}$ assuming, without loss of generality, that the underlying set $N=\omega_1$.  Though the variables are assumed to be distinct, the interpretations of the variables are not required to be distinct:

\begin{multline}
\Gamma((x_\alpha)_{ \alpha \in \omega_1 + k }) :=\bigcup_{\ov{\alpha} \in {\omega_1}^p} \{\varphi(x_{\alpha_0},\ldots,x_{\alpha_{p-1}}) :  \\
 \varphi(\ov{x}_{\ov{\alpha}})\textrm{~a literal in~} \LL,
\fNN \vDash \varphi(\alpha_0,\ldots,\alpha_{p-1}) \}\cup \bigcup_{s\in k}\{x_{\omega_1+s}=e_s\}.
\end{multline}

Now define another type that states that all (increasing) copies of $\AAA$ from the underlying set are assigned a color from $S_0$ by $f$:

$$\Sigma( (x_\alpha)_{\alpha\in \omega_1}) = \{ \bigvee_{s \in S_0} f(\ov{x}_{\ov{\alpha}}) = e_s : \ov{\alpha} \in {\omega_1}^n, \fNN \vDash \theta_{\AAA}(\ov{\alpha}) \}.
$$

Corollary \ref{minimal_saturation} together with Lemma \ref{technical} guarantee sufficient saturation of $\widehat{\fMM^*}$ to realize $\Gamma \cup \Sigma$, as long as $\Gamma \cup \Sigma$ is indeed a type finitely satisfiable in $\widehat{\fMM^*}$.
Any realization $\ov{a}$ of $\Gamma\cup \Sigma$ in $\widehat{\fMM^*}$ restricted to the initial segment of length $\omega_1$ is a copy $\fNN'$ of $\mathcal{\fNN}$ in $\fMM^*$.
Moreover, since $\ov{a}$ is a realization of $\Gamma\cup \Sigma$ in $\widehat{\fMM^*}$, the type $\Gamma\cup \Sigma$ implies the condition in \eqref{main_key}, as desired.

\

It remains to show that $\Gamma \cup \Sigma$ is finitely satisfiable in $\widehat{\fMM^*}$.  Any finite subset $\Gamma_0$ of $\Gamma$ can be expanded to have the following form, for some finite substructure $\BB \subseteq \fNN$:

$$\Gamma_0((x_\alpha)_{\alpha \in B\cup [\omega_1, \omega_1+k)}) = \{\varphi(x_{\alpha_0},\ldots,x_{\alpha_{p-1}}) : \varphi(\ov{x}_{\ov{\alpha}})\textrm{~a literal in~} \LL,  \BB \vDash \varphi(\ov{\alpha}) \}$$ 
$$\cup \; \{ \bigvee_{s \in S_0} f(\ov{x}_{\ov{\alpha}}) = e_s : \ov{\alpha} \in B^n, \BB\vDash \theta_{\AAA}(\ov{\alpha})\}\cup \bigcup_{s\in k}\{x_{\omega_1+s}=e_s\}.$$

\noindent The substitution of $\BB$ for $\fNN$ in $\Gamma_0$ is valid since $\BB \subseteq \fNN$ and $\theta_\AAA(\ov{x})$ is quantifier-free.
Since $\BB \in \age(\fNN) = \K$, there is some $i \in \omega$ such that $\BB \hookrightarrow \BB_i$.  Since $G \subseteq \omega$ is infinite, it must be cofinal in $\omega$, so there exists $m \in \omega\setminus i$ such that $m \in G$.  Note that  
\begin{eqnarray}\label{main4}
\BB \hookrightarrow \BB_m,
\end{eqnarray}
because $\BB \hookrightarrow \BB_i \hookrightarrow \BB_m$.
Since $m\in G$, $S(\BB_m) = S_0$, so $\widehat{\fMM^*} \vDash \varphi_{\BB_m,S(\BB_m)}$ implies that $\widehat{\fMM^*} \vDash \varphi_{\BB_m,S_0}$.
Thus, there exists a copy $\BB'_m$ of $\BB_m$ in $\fMM^*$ such
 that 
 \begin{eqnarray}\label{main3}
 {f^{\widehat{\fMM^*}}\;} ''  {\BB'_m \choose \AAA} \subseteq \{e^{\widehat{\fMM^*}}_s : s \in S_0\}.
 \end{eqnarray}
Define $\widehat{\BB'}_m$ to be the
$\widehat{\LL}$-substructure of $\widehat{\fMM^*}$ on $\BB'_m \cup \{e_j^{\widehat{\fMM^*}} : j \in k\}$ (not necessarily a disjoint union). 
We can use $\widehat{\BB'}_m$ to satisfy $\Gamma_0\cup \Sigma$ 
by \eqref{main4} and \eqref{main3} and by Remark \ref{strongerfortheta}.
$\eop_{\ref{fabulous}}$
\end{proof}

Recall Definition \ref{internalbRd} of internal big Ramsey degree.
The following corollary of Theorem \ref{fabulous} provides a partial converse to Observation \ref{big-implies-small}, transferring finiteness of small Ramsey degree to that of internal big Ramsey degree in ultraproducts of $\DD$-trending sequences, assuming CH.
\begin{corollary}[CH]\label{CHcorollary} 
Let $\LL$, $\DD$, $(\MM_i)_{i \in \omega}$ and $\fMM^*$ be as in the statement of Theorem \ref{fabulous}, and let $\AAA \in \age(\fMM^*)$. Suppose $\AAA$ has finite small Ramsey degree in $\age(\fMM^*)$. Then $\AAA$ has finite internal big Ramsey degree in $\fMM^*$, and \[ T_\mathrm{int}(\AAA, \fMM^*) \le t(\AAA, \age(\fMM^*)).\] 
\end{corollary}
\begin{proof} By Theorem \ref{sizec} and the Continuum Hypothesis, we can
take $\fNN = \fMM^*$ in Theorem \ref{fabulous}, so the conclusion follows by applying Observation \ref{agecontain} and Theorem \ref{fabulous}.
$\eop_{\ref{CHcorollary}}$
\end{proof}

\section{Big Ramsey degrees for internal colorings: general case}\label{general case}
We now show that the results from the previous section, let us call it the countable case, hold more generally in the case of ultraproducts over an ultrafilter on an arbitrary infinite cardinal $\kappa$. 
Although the countable case can of course  be obtained as a special case of the general one, we decided to give a direct exposition of the countable case from the previous section first, to improve readability.

In the general case of Theorem \ref{fabulous} we shall rely on the following new notion to replace that of $\DD$-trending:

\begin{definition}\label{ctbly_conforming} Let $\kappa$ be an infinite cardinal and let $\DD$ be a countably incomplete ultrafilter over $\kappa$. Suppose that $\KK$ is the age of an infinite structure in a finite relational signature such that the countably many isomorphism types from $\KK$ are enumerated in a sequence $(\AAA_n)_ {n \in \omega}$ satisfying the following conditions
\begin{enumerate}
    \item $\AAA_n \in \KK$, for all $n \in \omega$,
    \item $\AAA_n \hookrightarrow \AAA_{n+1}$, for all $n \in \omega$,
    \item for every $\AAA \in \KK$, there exists $n \in \omega$ such that $\AAA \hookrightarrow \AAA_n$.
\end{enumerate}
  A sequence $(\BB_i)_{i \in \kappa}$ is said to be \textit{$\DD$-countably-conforming-from-$\KK$, as witnessed by $(X_n)_{n \in \omega}$} if 
there is some sequence of sets $(X_n)_{n \in \omega}$ witnessing that $\DD$ is countably
incomplete and such that for every $i \in \kappa$, we have
$\BB_i :=\AAA_n$ for the first $n$ such that $i \notin X_n$.
\end{definition}

\begin{observation}\label{agecontain2} Suppose that $\LL$ is a finite relational signature. Then there is sequence of finite structures $(\BB_i)_{i \in \kappa}$ as in Definition \ref{ctbly_conforming} and which is definable from $\DD$, $(\AAA_n)_{ n \in \omega}$, and $(X_n)_{n \in \omega}$.

Let $\fMM^*:=\prod_{i \in \kappa} \BB_i/\DD$.  Then $\age(\fMM^*)=\KK:=\bigcup \{\age(\BB_i) : i \in \kappa \}$.
\end{observation}

\begin{proof} Since $(X_n)_ {n \in \omega}$ witnesses the countable incompleteness of $\DD$, each $X_n \in \DD$ and $\bigcap_{n \in \omega} X_n = \emptyset$.  So, the sequence $(\BB_i)_ {i \in \kappa}$ is well-defined.

For the second claim, $\age(\fMM^*) \subseteq \KK$ and $\KK = \bigcup \{\age(\BB_i) : i\in \kappa \}$ are clear from arguments in Section \ref{omegacase}.  To see that $\KK \subseteq \age(\fMM^*)$, we fix some notation.  Let $Y_n :=\bigcap_{m \in n} X_m \setminus X_n$. Since each $X_n\in \DD$, we have 
$\kappa \setminus X_n \notin \DD$, so $Y_n \notin \DD$.  
However,
$\bigcup_{n \in \omega} Y_n = \bigcup_{n \in \omega} (\kappa\setminus X_n) = \kappa \in \DD$. In particular, for every $n$, we have $T_n:=\bigcup_{m \ge n} Y_m\in \DD$,
because $\DD$ is an ultrafilter.  

We now note that $\bigcup \{\age(\BB_i) : i \in \kappa \}=\bigcup  \{\age(\AAA_n) : n \in \omega\}$. On the other hand, 
given $\AAA\in \KK$, there is $n<\omega$ such that $\AAA \hookrightarrow \AAA_n$ and hence $\AAA \hookrightarrow \AAA_m$ for all $m\ge n$.  Then $\AAA_n \hookrightarrow \BB_i$ for all $i \in T_n$ and $T_n \in \DD$.  This shows that $\KK \subseteq \age(\fMM^*)$.
$\eop_{\ref{agecontain2}}$
\end{proof}

\begin{remark} If we specialize to the case where $\kappa=\omega$, $\DD$ is any nonprincipal ultrafilter on $\omega$, and $X_n :=\{m \in \omega : n < m\}$, for $n \in \omega$, then $(X_n)_{n \in \omega}$ witnesses the countable incompleteness of $\DD$.  Moreover, $Y_n = \{ n \}$, in the notation of Observation \ref{agecontain2}. Thus, any sequence $(\BB_i)_{i \in \omega}$ of finite structures in a finite relational signature
 that is $\DD$-countably-conforming-from-$\KK$ witnessed by $(X_n)_{n \in \omega}$ is actually $\DD$-trending, by Observation \ref{canbedone}.
\end{remark}

\begin{theorem}\label{fabulous_general}
Suppose that $\kappa$ is an infinite cardinal, $\DD$ a $\kappa^+$-good countably incomplete ultrafilter over $\kappa$ and
$\LL$ a finite relational signature. Further, let $\KK_1$ be an age of $\LL$-structures.

Let $(\MM_i)_{i\in\kappa}$ be a sequence of finite $\LL$-structures $\DD$-countably-conforming-from-$\KK_1$ (witnessed by some sequence $(X_n)_{n \in \omega}$)  and define $\fMM^*:=\prod_{i \in \kappa} \MM_i/\DD$. 

Let $\fNN$ be an $\LL$-structure of cardinality $\kappa^+$ such that
$\KK:=\age(\fNN) \subseteq \bigcup \{\age(\MM_i) : i \in \kappa \}$

Fix $\AAA \in \K$ and suppose that $\AAA$ has finite small Ramsey degree in $\K$, denoted $t := t(\AAA, \K)$. Then for any $k \in \omega$, the partition relation $ \fMM^* \lrawi \bigl( \fNN \bigr)^\AAA_{k, t}$ holds.
\end{theorem}

\begin{proof} We know that $\KK_1 = \age(\fMM^*)$ by Observation \ref{agecontain2}.
Thus,
Theorem \ref{maximal_saturation} together with Lemma \ref{technical} can be used to follow a similar argument as in Theorem \ref{fabulous}.  Up until equation \eqref{1main}, where we conclude that ``$\fMM^* \vDash \varphi_{\BB,S(\BB)}$'', the argument is the same, except that $\{i \in \kappa : C_{\BB} \hookrightarrow \MM_i\} \in \DD$ follows from Observation \ref{agecontain2}.  Instead of working with the sequence $(\BB_i)_{i \in \omega}$ as in Theorem \ref{fabulous}, now we choose to substitute the sequence $(\MM_i)_{i\in \kappa}$.  Let $Y_n:=\bigcap_{m \in n} X_m \setminus X_n$ and define $T_n:=\bigcup_{n \leq i} Y_i\in \DD$ as in Observation \ref{agecontain2}.  There is some integer $\ell$ such that $(S_r)_{r \in \ell}$ enumerates the finite subsets of $k$ of size $t$.  In analogy with the map $\BB \mapsto S(\BB)$, we have a map $\MM_i \mapsto S(\MM_i)$ such that 
$$\fMM^* \vDash \varphi_{\MM_i,S(\MM_i)}$$
that yields a finite partition $(P_r)_{r\in \ell}$ of $\kappa$ as follows: 
$$P_r:= \{i \in \kappa : S(\MM_i)=S_r\}.$$
For every $n \in \omega$, the tail $T_n \subseteq \kappa$ intersects at least one piece of the partition, choose one, $P_{r(n)}$, so we have a map $n \mapsto r(n)$ from $\omega$ into $\ell$.  By the pigeonhole principle, there exists $r_0 \in \ell$ such that $\{n \in \omega: r(n)=r_0\} \subseteq \omega$ is infinite, and thus, is cofinal in $\omega$.

Write the same types $\Gamma \cup \Sigma$, but in $\kappa^+$-many variables, and substitute the set $S_{r_0}$ for the set $S_0$ named in the proof for Theorem \ref{fabulous}.  In the same way, we are left with a finite structure
$\BB$ associated to an arbitrary finite subtype, $\Gamma_0$.  Instead of writing ``$\BB \hookrightarrow \BB_i \hookrightarrow \BB_m$'' we argue in the following way:

\vspace{.1in}
\noindent For every $\BB \in \age(\fNN) \subseteq \KK_1$, there exists some $n^* \in \omega$ such that $\BB \hookrightarrow \MM_i$, for all $i \in T_{n^*}$ (for example, take $n^*$ such that $\BB=\AAA_{n^*}$, as in the proof of Observation \ref{agecontain2}). In addition, there exists $m^* \ge n^*$ such that $r(m^*)=r_0$, by cofinality in $\omega$.  However, $\BB \hookrightarrow \MM_i$ for all $i \in T_{m^*} \subseteq T_{n^*}.$  Now take any $i^* \in T_{m^*} \cap P_{r(m^*)} \neq \emptyset$, this index plays the role of $m$ at the end of Theorem \ref{fabulous}.  Since $i^* \in P_{r(m^*)}$ and $r(m^*)=r_0$, we know that $S(\MM_{i^*}) = S_{r(m^*)}=S_{r_0}$.  Thus, there is a copy $\MM'_{i^*} \cong \MM_{i^*}$ in $\fMM^*$ whose copies of $\AAA$ in $\fMM^*$ take on only colors in $S_{r_0}$ under $c^*$.  Moreover, since $i^* \in T_{m^*}$, $\BB \hookrightarrow \MM_{i^*}$, as desired.  
$\eop_{\ref{fabulous_general}}$
\end{proof}

\begin{corollary}[GCH]\label{GCHcorollary} 
Let $\kappa$, $\LL$, $\DD$, $(\MM_t)_{t \in \kappa}$, $\KK_1$ and $\fMM^*$ be as in the statement of Theorem \ref{fabulous_general}, and let $\AAA \in \KK_1$. Suppose $\AAA$ has finite small Ramsey degree in $\age(\fMM^*)$. Then $\AAA$ has finite internal big Ramsey degree in $\fMM^*$, and \[ T_\mathrm{int}(\AAA, \fMM^*) \le t(\AAA, \age(\fMM^*)).\] 
\end{corollary}
\begin{proof}
Take $\fNN = \fMM^*$ in Theorem \ref{fabulous_general}, applying Fact \ref{sizegc} and the Generalized Continuum Hypothesis.  
$\eop_{\ref{GCHcorollary}}$
\end{proof}

\section{The case of linear orders under CH}\label{eta1}

All infinite ultraproducts of a countable sequence of finite linear orders are elementarily equivalent and are of the form
	\begin{eqnarray}\label{ult1}
	\omega+\sum_{\mathfrak{L}}(\omega^*+\omega)+\omega^*,
	\end{eqnarray}
	where $\omega^*$ is the reverse of the natural linear order on $\omega$ and $\mathfrak{L}$ is an $\aleph_1$-saturated dense linear order without endpoints. This fact appears in many resources; for a proof, see e.g. \cite{Garcia}.

	In this section, we focus on the concrete example of linear orders with the additional assumption of CH. 
	Under CH, by Theorem \ref{uniquemodel}, there is a unique up to isomorphism $\aleph_1$-saturated dense linear order without endpoints and of cardinality $\aleph_1$, which we denote by $\fEE$.  Thus, under CH, there is a unique up to isomorphism infinite ultraproduct $\mathfrak{L}^*$ of a countable sequence of finite linear orders, since this $\fEE$ must play the role of the ``spine'' in \eqref{ult1}:
	\begin{eqnarray}\label{ult2}
	\mathfrak{L}^*=\omega+\sum_{\fEE}(\omega^*+\omega)+\omega^*.
	\end{eqnarray}
	One way to see this is to note that by 
	Theorem \ref{sizec}, $\mathfrak{L}^*$ is of size $\aleph_1$, and thus the order $\mathfrak{L}$ in \eqref{ult1} is not only $\aleph_1$-saturated and dense without endpoints, but of size $\aleph_1$.
	
	A basic observation is that every copy of $\mathbb{Q}$ in $\mathfrak{L}^*$ takes at most one point in any copy of $\omega^*+\omega,$ and thus is in fact a copy of $\mathbb{Q}$ in $\fEE$.  We will use this fact in the proofs of both Corollary \ref{maincounterexample} and Corollary \ref{maincounterexample2}, which are counterexamples to Corollary \ref{CHcorollary} in the case that the colorings are not internal.
	
	\subsubsection*{Trees}
	Just as the rational linear order $\mathbb{Q}$ can be represented on the binary tree $2^{<\omega}$ with a modified lexicographic order, under CH we can view $\fEE$ as the full binary tree $2^{<\omega_1}$ with analogously defined linear order. 
	To do so, we introduce some notation around trees.  	A \emph{tree} is a partially ordered set $(T,\leq)$ such that for every $t\in T,$ the set of predecessors of $t$, $\text{pred}(t)=\{s\in T:s< t\}$ is well ordered. For $t\in T,$ we denote by $|t|$ the \emph{height} of $t$, that is, the order type of $\text{pred}(t)$. The height of $T$, denoted by $\text{ht}(T),$ is $\text{sup}_{t\in T} |t|.$ For an ordinal $\alpha\leq \text{ht}(T)$, the $\alpha$-th level of $T$, is $T(\alpha)=\{t\in T: |t|=\alpha\}.$
	If $T$ has a minimal element, it is referred to as the \emph{root} of $T$ and is denoted by $\text{root}(T)$. A \emph{terminal node} of $T$ is a node without a successor. 
	A \emph{subtree} of $(T,\leq)$ is any subset $U\subseteq T$  with the induced partial order. Notice, that any subtree is automatically a tree. A subset $X\subseteq T$ is an \emph{antichain}, if its elements are pairwise incomparable.

	 In what follows, we will be working with the full binary rooted tree $(2^{<\omega_1},\sqsubseteq)$, where $s\sqsubseteq t$ if $s$ is an initial segment of $t$. Finite sequences such as $\la 0, 1\ra$ may be denoted by $01$ for clarity.  We denote by $\left<\right>$ the empty sequence, which is the root of $2^{<\omega_1}.$ Recall that for $s \in 2^{<\omega_1}$, $|s|$ denotes the length of $s$ as a sequence, which coincides with the height of $s$ in $2^{<\omega_1}$. For every $s$ and $t$ in $2^{<\omega_1}$, there exists a unique greatest lower bound of $s$ and $t$ denoted by $s\wedge t$, which is the longest common initial segment of $s$ and $t$. By a subtree of $2^{<\omega_1}$, we will always mean a rooted subtree. A subtree $T$ of $2^{<\omega}$ is \emph{perfect} if it has no terminal nodes.

	 Let  $<_{\text{lex}}$ denote the usual lexicographical order on $\sqsubseteq$-incomparable elements $s, t \in 2^{<\omega_1}$, i.e. $s<_{\textrm{lex}} t$ if $(s \wedge t)^\smallfrown  0  \sqsubseteq s$ and $(s \wedge t)^\smallfrown  1  \sqsubseteq t$.  We define a linear order $<_E$ on $2^{<\omega_1}$ by $s\leq_E t$ if either $s=t$, or $s\sqsubseteq t$ and $s^{\smallfrown}1\sqsubseteq t$, or $t\sqsubseteq s$ and $t^{\smallfrown}0\sqsubseteq s$, or $s$ and $t$ are $\sqsubseteq$-incomparable and $s<_{\text{lex}}t.$  Note that $<_{\text{lex}}$ and $<_E$ coincide on any antichain in $2^{<\omega_1}.$  It is known that $(2^{<\omega_1},<_E)$ is an $\aleph_1$-dense linear order without endpoints, so we have the following under CH:
	 \begin{proposition}[CH]
	$(2^{<\omega_1},<_E) \cong \fEE$. 
	 \end{proposition}

	\subsection{Some history around $\eta_1$ and partition properties}\label{6.1}

The linear order $\fEE$ introduced above is an $\eta_1$
set as introduced by Hausdorff (see pg. 488 of \cite{Hausdorffeta}). More generally, if $\alpha$ is an ordinal, an $\eta_{\alpha}$ set is a linearly ordered set $X$ such that for every two subsets $Y,Z$ of cardinality less than $\aleph_{\alpha}$ with every element in $Y$ below every element of $Z,$ there is $x\in X$ that is above everything in $Y$ and below everything in $Z$. There is unique $\eta_0$ set of size $\aleph_0,$ which is the order type of the rationals, $(\mathbb{Q},\leq)$, and it is typically referred to as the order type $\eta$. For $\alpha\geq 1$, an $\eta_{\alpha}$ set of cardinality $\aleph_{\alpha}$ may or may not exist, depending on the additional set theoretic axioms. If it exists, then it is unique up to isomorphism.

	For a cardinal $\kappa$ satisfying $\kappa^{<\kappa}=\kappa$, one can define J\'onsson limits of classes of structures of size $<\kappa$, in analogy to Fra\"iss\'e limits, giving a J\'onsson limiting structure of size $\kappa$, see Chapter IV of \cite{ComfortNeg974} for an exposition. Under CH, $\aleph_1^{<\aleph_1}=\aleph_1$ and the order $\fEE$ is the J\'onsson limit of countable linear orders. Hence it is in particular an $\aleph_1-$universal and $\aleph_1$-homogeneous linear order of size $\aleph_1$ and it is, moreover, saturated.

There is much literature on partition relations for uncountable cardinals, and we will highlight only a few references here.
In \cite{MR1556708} it is shown by Sierpi\'{n}ski that 
$2^{\aleph_0} \not \rightarrow (\aleph_1)^2_2$ and in
\cite{Kurepa59} by Kurepa that
$2^\kappa \not \rightarrow (\kappa^+)^2_2$.
This work sets the foundation for a theory of big Ramsey degrees for infinite cardinals. In Theorem 17A of \cite{EHR} it is shown by Erd\H{o}s, Hajnal and Rado that, under GCH, \\
$\aleph_{\alpha+1} \not \rightarrow (\aleph_{\alpha+1})^2_{\aleph_{\alpha+1}}$,
by a method assuring that subsets of size $\aleph_{\alpha+1}$ contain pairs from every piece of the partition.  Galvin and Shelah in \cite{GalvinShelah73} and Todor\v{c}evi\'{c} in \cite{Tod87} give a history of contemporary developments that have influenced their ZFC results which we review below.  More recently, the paper \cite{MaSo21} gives an account of big Ramsey degrees for all countable ordinals.  

It is shown in \cite{GalvinShelah73} that 
$2^{\aleph_0} \not \rightarrow (2^{\aleph_0})^2_{\aleph_0}$, and the witnessing coloring is such that if $2^{\aleph_0}$ is regular, then in fact subsets of size $2^{\aleph_0}$ have pairs from every piece of the partition. This will be useful in our Corollary \ref{maincounterexample}.
In \cite{Tod87} it is shown that there exists a coloring $d: {[\omega_1]}^2 \rightarrow \omega_1$ such that for any uncountable set $D \subseteq \omega_1$, $d''[D]^2=\omega_1$.

In the following, we will explain the consequences for $\mathfrak{L}^*$.  Since we are ultimately working under CH, we could refer to any of the results above in \cite{EHR},\cite{GalvinShelah73},\cite{Tod87} for the following Lemma \ref{omega1}, but since the lemma can be stated without the hypothesis of CH, we do so.  Given an infinite linear order $\mathfrak{L}$ and $m \in \omega$, we  identify $[\mathfrak{L}]^m$ with increasing $m$-tuples from $\mathfrak{L}$, which in turn may be identified with ${\fLL \choose \AAA}$, where $\AAA$ is the unique linearly ordered structure with $|A| = m$.  In this way we may say the ``big Ramsey degree of $m$-tuples in $\fLL$'' to refer to the big Ramsey degree of $\AAA$ in $\fLL$.

	\begin{lemma}[\cite{Tod87}]\label{omega1} For every $m \geq 2$, $m$-tuples have infinite big Ramsey degree in $(\omega_1,\in)$.
	\end{lemma}
	
	\begin{proof} 
	Fix a coloring $d: {[\omega_1]}^2 \rightarrow \omega_1$ as in \cite{Tod87} such that for any uncountable set $D \subseteq \omega_1$, $d''[D]^2=\omega_1$.  For any $k, m \in \omega$, and set $A \in {[\omega_1]}^m$ listed as $\ov{a}$ in increasing order, define $c(A):=d(\{a_0,a_1\})$, if $d(\{a_0,a_1\}) \in k$ and $c(A):=0$, otherwise.  Thus, $d$ induces a $k$-coloring $c: {[\omega_1]}^m \rightarrow k$ such that for any uncountable set $D \subseteq \omega_1$, $c''[D]^m=k$.  
	$\eop_{\ref{omega1}}$
	\end{proof}

By Corollary \ref{CHcorollary}, the colorings witnessing the result in Corollary \ref{maincounterexample} are necessarily external, as otherwise the big Ramsey degree of $m$-tuples in the ultraproduct would be bounded by $d=1$, by the Finite Ramsey Theorem.

	\begin{corollary}[CH]\label{maincounterexample}
		For every $m \geq 2$, $m$-tuples have infinite big Ramsey degree in  $\mathfrak{L}^*$.
	\end{corollary}
\begin{proof}
It suffices to prove the result for $\fEE$ instead of $\fLL^*$. Fix $m, k \in \omega$ such that $m, k \geq 2$.  It suffices to show that for such an arbitrary $k$ there exists a $k$-coloring $f: [E]^m \rightarrow k$ such that for any $\fEE' \subseteq \fEE$ such that $\fEE' \cong \fEE$, $f'' [\fEE']^m = k$.

Identify $\fEE$ with $(2^{<\omega_1},<_{E})$ as before.  
Let $c: {[\omega_1]}^m \rightarrow k$ be as in the proof of Lemma \ref{omega1}.  For any $A \in [2^{<\omega_1}]^m$, define $f(A)=c(\{|a| : a \in A\})$, if $|a| \neq |b|$ for all $a\neq b$ such that $a,b \in A$, and $f(A) = 0$, otherwise.  Let $\fEE' \subseteq \fEE$ be any copy of $\fEE$ (in the signature $\{<\}$) and let $H:=\{|a| : a \in \fEE'\}$.  We have that $H \subseteq \omega_1$ is uncountable (since $\fEE'\subseteq 2^{<\omega_1}$ is uncountable and $<_E$ and $<_{\text{lex}}$ (which is not saturated) agree on the antichains of  $2^{<\omega_1}$). Thus $c''[H]^m = k$.  By definition of $H$, for every set $B \in [H]^m$, there exists a set $A \in [\fEE']^m$ such that $\{|a| : a \in A \}=B$.  Thus $k = c''[H]^m \subseteq f''[\fEE']^m$, as desired.
$\eop_{\ref{maincounterexample}}$
\end{proof}

\subsection{Witnessing Devlin types in all copies of $\eta$ in $\mathfrak{L}^*$}\label{6.2}

We will show in Corollary \ref{maincounterexample2} that the coloring of finite tuples of the Cantor tree $2^{<\omega}$ by Devlin embedding types as in \cite{Vuksanovic} and \cite{MR2603812}, extended to $2^{<\omega_1}$, witnesses the optimal failure of Corollary \ref{CHcorollary} in the case that colorings are allowed to be external.

	\subsubsection*{Devlin types}
	
	For an $n$-tuple $A$ in $2^{<\omega_1}$, we denote by $A^{\wedge}$ the meet subtree 
	\[
	A^{\wedge}=\{s\wedge t:s,t\in A\}. 
	\]
	Note that if $A$ is the set of terminal nodes of $A^{\wedge}$ of size $n,$ then $A^{\wedge}$ has $2n-1$ nodes.  
	We include the following definitions of (Devlin) embedding types from \cite{MR2603812} (Definition 6.11 and Definition 6.19) in $2^{<\omega_1}$ for completeness.

	\begin{definition}
		Let $A,B$ be finite subsets of $2^{<\omega_1}$. We say that $A$ and $B$ have the same embedding type in $2^{<\omega_1}$ if there is a  tree isomorphism $f:A^{\wedge}\to B^{\wedge}$ that sends $A$ to $B$, respects height, i.e. $|s|<|t|$ if and only if  $|f(s)|<|f(t)|$ for every $s,t\in A^{\wedge}$, and $t(|s|)=f(t)(|f(s)|)$ whenever $|s|<|t|.$
	\end{definition}
    
    Clearly, having the same embedding type is an equivalence relation on $n$-tuples of $2^{<\omega_1}$, for every finite $n.$
	
	\begin{definition}
		A finite set $A\subseteq 2^{<\omega_1}$ of size $n\geq 1$ realizes a Devlin embedding type if
		\begin{enumerate}
			\item $A$ is the set of terminal nodes of $A^{\wedge}$,
			\item $|s|\neq |t|$ for any $s\neq t$ in $A^{\wedge}$,
			\item $t(|s|)=0$ for all $s,t\in A^{\wedge}$ with $|s|<|t|$ and $s\not\sqsubseteq t$. 
			\end{enumerate}
	\end{definition}
    
    The number of Devlin embedding types for $n\geq 1$ is $T_{2n-1}$, the $(2n-1)$-st tangent number, as shown in \cite{MR2603812}. Devlin proved that these are exactly the big Ramsey degrees of $\eta,$ the order type of rationals.
    
    \begin{theorem}[Devlin \cite{Devlin}]\label{Devlin}
        Let $n\geq 1$, $k\geq 2$, and let $t_n=T_{2n-1}$. Then 
        \[
        \eta\lraw (\eta)^n_{k,t_n}.
        \]
    \end{theorem}

	\subsubsection*{Construction}\label{cons_ce}
	We will build a skew subtree $W$ of $2^{<\omega_1}$ isomorphic to $2^{<\omega_1},$ and thus order isomorphic to $\fEE$ with the order induced by $<_{E}$. The desired copy of $\fEE$ will be obtained as an antichain $X$, a slight tweak of $W$. We will show that every $Y\subseteq X$ that is order isomorphic to $(\mathbb{Q},<)$ realizes every Devlin embedding type.  The technique was introduced in \cite{Vuksanovic} for $2^{<\omega}$. 
	The first $\omega$-levels of $W$, denoted by $W_0$, will be identical to $S$ in the proof of Lemma 6.20 in \cite{MR2603812}. We list its defining properties here: 
	\begin{enumerate}
		\item $\text{root}(W_0)=\left<\right>$,
		\item $|W_0\cap 2^{3n}|=1$ and $|W_0\cap 2^{3n+1}|=|W_0\cap 2^{3n+2}|=0$ for all $n\in \omega$,
		\item $W_0$ is isomorphic to $2^{<\omega}$,
		\item for every $m\in\omega$, if $s,t\in W_0(m)$ and $s<_E t$, then  $|s|<|t|$
		\item for every two natural numbers $m \in n$, $s\in W_0(m),$ and $t\in W_0(n),$ we have $|s|<|t|,$
		\item for any $s\in W_0$ and $t\in 2^{<\omega}\setminus W_0$, if $t\sqsubseteq s$, then $t^{\frown}0\sqsubseteq s$.
		
	\end{enumerate}
	
	 On $\omega$-th level, we run into the difficulty that from now on the levels are as large as the height of the tree. Since we will require that on every level of $2^{<\omega_1}$, $W$ takes  at most one node as in item (2), we cannot achieve that the heights of elements in $2^{<\omega_1}$ that form the $\omega$-th level of $W$ are strictly less than the height of any element in $2^{<\omega_1}$ on $(\omega+1)$-st level of $W$ as in item (5), or that the $<_E$ order determines the height order as in item (4). 
	 We will modify the construction as follows: Every branch through $W_0$ will be extended to some leftmost successor that will serve as the root of a copy of $W_0$ forming the next $\omega$ levels of $W$, and we repeat for limit points that appear, and so on.  Every limit point will be taken care of once but enumerated uncountably many times. Here, by a \emph{limit point} of a subtree $T \subseteq 2^{<\omega_1}$ we mean a point $x\in 2^{<\omega_1} \setminus T$ such that $x \uphp \alpha\in T$ for every $\alpha\in|x|$.  Below, the technical details of the construction are carried out.

	 We fix a bookkeeping bijection $b:\omega_1\times\omega_1\to \omega_1$ such that $b(\alpha,\beta)\geq \alpha$ and proceed by recursion along $\omega_1$. We let $x(0,\beta)=\langle \rangle$ for $\beta\in\omega_1$ and we let $W_0$ be as above. Suppose that we have constructed subtrees $W_{\gamma}$ of $2^{<\omega_1}$ for $\gamma\in \delta \in \omega_1$ and $X_\alpha:=\{x(\alpha,\beta):\beta\in \omega_1\}$ for all $\alpha \in \delta$ such that $W_{\gamma}\subseteq W_{\gamma'}$ for $\gamma \in \gamma'$ and $X_\alpha$ enumerates the limit points of $\bigcup_{\gamma\in\alpha}W_{\gamma}$ in $2^{<\omega_1}$.  Let $\alpha,\beta\in\omega_1$ be such that $b(\alpha,\beta)=\delta.$ Since $\delta=b(\alpha,\beta) \geq \alpha$, $x(\alpha,\beta) \in X_\alpha$ has already been defined.  If $x(\alpha,\beta)$ is in the downward closure of $ \bigcup_{\gamma\in\delta}W_{\gamma},$ we let $W_{\delta}=\bigcup_{\gamma\in\delta}W_{\gamma}.$  Otherwise, we let $W_{\delta}$ be the union of $\bigcup_{\gamma\in\delta}W_{\gamma}$ with the subtree of $2^{<\omega_1}$ rooted at the leftmost successor of $x(\alpha,\beta)$, $x(\alpha,\beta)^\smallfrown \langle 0, \ldots, 0\rangle$, on the level $\omega\cdot\delta$ and isomorphic to $W_0$ according to a $<_{\textrm{lex}}$- and level-preserving map that identifies $\{\omega\cdot\delta+n:n\in\omega\}$ with $\omega$. We let $W=\bigcup_{\gamma\in\omega_1}W_{\gamma}.$

	 Finally, we define $X=\{w^{\frown} 01: w\in W\}.$ Clearly, $X$ is an antichain in $2^{<\omega_1}$ order isomorphic to $\fEE$ and every finite tuple $A\subseteq X$ realizes a Devlin embedding type. (Note that $<_{\text{lex}}$ and $<_E$ coincide on $X,$ since it is an antichain, and its Devlin embedding types are determined by $\sqsubseteq$-isomorphisms preserving height and $<_{\text{lex}}$).

	\begin{lemma}\label{Devlinlemma}
		Every copy of $(\mathbb{Q},<)$ in $X$ contains every Devlin embedding type. 
	\end{lemma}

	\begin{proof}
		Let $Y\subset X$ be a countable set that contains a copy of $\mathbb{Q}$. It means that $Y^{\wedge}$ contains a perfect subtree $U$. For $s\in U$ and $i\in \{0,1\}$, let $U_i^s=\{t\in U:s^{\frown} i\sqsubseteq t\}$. Let $\text{sup}_i^s$ denote $\text{sup}\{|t|:t\in U_i^s\}.$ Unlike in the case of $2^{<\omega},$ we may have $\text{sup}_0^s\neq\text{sup}_1^s.$ However, we will find a perfect subtree $U'$ of $U$ satisfying a condition that for every $s\in U',$ $\text{sup}_0^s=\text{sup}_1^s$. Suppose not. Let $s_0\in U$ be such that, without loss of generality, $\text{sup}_0^{s_0} < \text{sup}_1^{s_0}$. Since $U_0^{s_0}$ does not satisfy the desired condition, there is $s_1\in U_0^{s_0}$ such that, without loss of generality, $\text{sup}_0^{s_1}<\text{sup}_1^{s_1}.$ We can continue recursively to construct a strictly decreasing sequence of ordinals $\text{sup}_0^{s_0}>\text{sup}_0^{s_1}>\ldots$. This is impossible and therefore the existence of $U'$ is secured. Since $U$ is binary, for every $s\in U,$ $\text{sup}\{|t|:t\in U \ \&\ t\sqsupseteq s\}=\text{max}\{\text{sup}_i^s:i=0,1\},$ and thus $\text{sup}_i^s$ is the same for every $s\in U'$ and $i=0,1.$ Let $\alpha$ denote this common supremum. As $U'$ is a perfect tree, we can recursively construct a perfect subtree $U''$ such that levels through every branch in $U''$ have the same supremum: Fix an increasing sequence $(\alpha_n)_{n\in\omega}$ of ordinals converging to $\alpha$. We let the root of $U'',$ $r_{\emptyset}$, be the root of $U'$. Since $\text{sup}_i^{r_{\emptyset}}=\alpha$ for $i=0,1,$ we can choose $r_i\in U_i^{r\emptyset}$ with $|r_i|\geq \alpha_1$. We recursively continue the construction for $(\alpha_n)_{n\in\omega}$ to obtain $U''$. Since $U''$ is perfect, we can further prune $U''$ to a perfect subtree $Z$ that satisfies items (1)--(6) with (2) replaced by consecutive levels of $Z$ being at least $3$ levels in $2^{<\aleph_1}$ apart. Clearly, there is an isomorphism $Z\to W_0$ preserving $<_E$. Since $W_0$ contains representatives of all Devlin embedding types (see \cite{MR2603812}, Lemma 2.20), so does $Z$. As in \cite{Vuksanovic}, (a) in the proof of Lemma 0.10, by additional pruning, we can further ensure that there is an injective function $f:Z\to Y$ such that $|f(z)|<|t|$ for every $z,t\in Z$ with $|z|<|t|$ and $z^{\smallfrown}0\sqsubseteq f(z).$ Therefore we can conclude that $Y$ contains tuples realizing every Devlin embedding type.  
		 $\eop_{\ref{Devlinlemma}}$
		\end{proof}
	
	\begin{corollary}[CH]\label{maincounterexample2}
 		Let $L_n$ be a finite linear order of size $n$ and let $t_n=T_{2n-1}$ be the $(2n-1)$-st tangent number. There exists a coloring (necessarily external) $c:{\mathfrak{L}^*\choose L_n}\to \{1,2,\ldots, t_n\}$ such that $c$ takes on all colors on every copy of $\mathbb{Q}$ in $\mathfrak{L}^*.$
		\end{corollary}
	
	\begin{proof}
		By the proof of Devlin's Theorem  \ref{Devlin} in \cite{MR2603812}, $t_n$ is the number of Devlin embedding types of $n$-tuples. Since $\mathbb{Q}$ is not scattered, every copy of $\mathbb{Q}$ in $\mathfrak{L}^*$ takes at most one element from any copy of $\omega^*+\omega$ in $\fEE$ and is therefore effectively a copy of $\mathbb{Q}$ in $\fEE$. The statement thus follows from Lemma \ref{Devlinlemma}.
		$\eop_{\ref{maincounterexample2}}$
		\end{proof}

	\section{Open problems and further directions}\label{openproblems}
	
	We present a few questions for future exploration.	The first question requires some background on universal minimal flows which we give here.
	Let $G$ be a topological group and $X$ a compact Hausdorff space. A continuous function $\alpha:G\times X\to X$ is a \textbf{$G$-flow} if 
	\begin{enumerate}
		\item $\alpha(e,x)=x$ for any $x\in X$ and $e$ the neutral element of $G$,
		\item $\alpha(gh,x)=\alpha(g,\alpha(h,x))$ for every $g,h\in G$ and $x\in X$.
	\end{enumerate}
	We will write $gx$ in place of $\alpha(g,x).$
	A $G$-flow on $X$ is \textbf{minimal} if $X$ does not contain a nonempty proper closed $G$-invariant subset. A \textbf{homomorphism} between $G$-flows $X$ and $Y$ is a 
	continuous map $\phi:X\to Y,$ such that for every $g\in G$ and $x\in X,$ we have $\phi(gx)=g\phi(x).$ If $\phi$ is onto, we say that $Y$ is a \textbf{quotient} of $X$ and if $\phi$ is bijective, it is called an \textbf{isomorphism}. Ellis showed that up to isomorphism, for every topological group $G$ there is a unique \textbf{universal minimal flow}, $M(G)$, that is, a minimal $G$-flow which has every minimal $G$-flow as a quotient.

	A structure $\fMM$ is called \textbf{$\omega$-homogeneous} if every finite partial isomorphism of $\fMM$ extends to an automorphism of $\fMM$.  The following result due to Pestov sets the stage for our inquiry into the universal minimal flows of automorphism groups of linear orders.  For the general case for countable structures see the  Kechris-Pestov-Todor\v{c}evi\'c correspondence from \cite{KPT}.

	\begin{theorem}[\cite{Pestov98}]
	Let $\fMM$ be an $\omega$-homogeneous linear order. Then $\textrm{Aut}(\fMM)$ is extremely amenable. 
	\end{theorem}
	
	The following lemma says that extreme amenability behaves well with respect to semi-direct product (actually, more generally, short exact sequences).
	The proof of the lemma is immediate.  
	\begin{lemma}\label{folklore}
		Let $G\cong H\ltimes K$ and suppose that $H$ is extremely amenable. Then $M(G)\cong M(K).$
		\end{lemma}
	
	As a result of Lemma \ref{folklore} and our analysis of the structure of $\mathfrak{L}^*$, we have the following.
	\begin{theorem}[CH]
	The group	$\textrm{Aut}(\mathfrak{L}^*)$ is isomorphic to
	$\textrm{Aut}(\fEE)\ltimes\mathbb{Z}^{\mathfrak{c}}.$ Consequently, $M(\textrm{Aut}(\mathfrak{L}^*))\cong M(\mathbb{Z}^{\mathfrak{c}}).$
		\end{theorem}
	
	\noindent This sets the stage for our first question.	
	\begin{question}
		What is the universal minimal flow of $\mathbb{Z}^{\mathfrak{c}}$?
	\end{question}	
	
	\noindent It would be intriguing if an analysis of $M(\textrm{Aut}(\mathfrak{L}^*))$ could lend tractability to this problem.
	
	A different line of inquiry is suggested by the following questions.

	\begin{question}
		What are the big Ramsey degrees of $\mathfrak{L}^*$ under CH with respect to other special types of colorings (such as colorings with certain topological or stability-theoretic properties)?
	\end{question}	

	\begin{question}
		What transfer principles can we obtain without CH?
		\end{question}
	
	\begin{question}
	    What happens if we equip the ultraproduct with a topology and consider continuous colorings of its finitely-generated substructures by the unit interval $[0,1]$?
	\end{question}

\section{Appendix: Saturation for types with many variables}\label{Appendix}

The goal of this appendix is to give a self-contained exposition of the argument for Lemma \ref{technical}.  The argument we give is an infinite version of the proof for Proposition \ref{ckprop} below.

\begin{proposition}[Proposition 2.3.6 of \cite{C-K}]\label{ckprop}  Let $\fAA$ be a structure that is $\aleph_0$-saturated (for 1-types).  For every $n \in \omega$ and for every $n$-type $p$ over a finite parameter set with respect to $\fAA$, $p$ is realized in $\fAA$.
\end{proposition}
 
Since we work with types in a transfinite sequence of variables, we use transfinite recursion.
For a reference on transfinite recursion see p. 21-22 of \cite{Jech-ST} or Ch.~III \S 5 of \cite{Kunen-ST}.  For a result that uses transfinite recursion in an analogous way, see Lemma \ref{ckprop2} below.

\begin{lemma}[Lemma 5.1.10 of \cite{C-K}]\label{ckprop2}  Suppose that $\fAA$ is $\alpha$-saturated, $\fAA \equiv \fBB$ and $b \in {B}^\alpha$.  Then there exists an $a \in {A}^\alpha$ such that $(\fAA,a_\xi)_{\xi\in\alpha} \equiv (\fBB,b_\xi)_{\xi\in\alpha}$.
\end{lemma}

\vspace{.1in}

We start with a restatement of the lemma followed by a proof.

\begin{ShLemma}\label{technical2+}Let $\lambda$ be an infinite cardinal and $\fMM$ a $\lambda$-saturated structure.  Suppose $p$ is a type over $M$ in free variables $\ov{x}$ with respect to $\fMM$  
such that $|\ov{x}| \leq \lambda$
and such that $|\Domp(p)|< \lambda$. Then $p$ is realized in $\fMM$.
\end{ShLemma}

\begin{proof} If $\lambda \in \aleph_0$, then we may follow the proof of Proposition \ref{ckprop} above.  So let us assume that $\lambda \geq \aleph_0$.
Let $X = (x_i)_{i\in \lambda}$ be a sequence of distinct variables.  Fix a $\lambda$-saturated structure $\fMM$ and a type $p(X)$ such that for some set $A \subseteq M$, $\textrm{Dom~} p = A$ and $|A|< \lambda$.  We wish to show that $p$ is realized in $\fMM$.

\

Given an ordinal $\beta \in \lambda$, define $$X_\beta := ( x_i)_{ i  \in \beta} \textrm{ and }Y_\beta := (x_i)_{ i \in \lambda \setminus \beta}.$$ 
Note that $X_{\beta+1} = {X_\beta}^\smallfrown \la x_\beta\ra$ and ${X_\beta}^\smallfrown Y_\beta = X$ for all $\beta \in \lambda$.

Define
$$\displaystyle q_\beta(X_\beta):= \{ \exists \ov{y} \bigwedge_{\varphi \in F} \varphi(X_\beta; \ov{y}) : F=F(X_\beta;\ov{y}) \subset_\omega p, \; \ov{y} \in Y_{\beta}^n, \;  n\in \omega \}.$$
Using AC we can assume there is a well-ordering $\prec$ on $M$.  
We shall attempt to define a function $G$ as follows:
for any set $Z$, 
if there exists
$\beta\in\lambda$ such that $Z:=( b_\gamma)_{\gamma \in \beta} \in M^\beta$ and $M \vDash q_\beta(Z)$, 
$$G(Z) := \textrm{~the~} \prec\textrm{-least element~} b \in M \textrm{~such that~} M \vDash q_{\beta+1}(Z^\smf \la b \ra);$$

\noindent otherwise, $G(Z) := \emptyset$.  

\vspace{.1in}

We first argue that $G$ is a well-defined function on $V$.

\begin{claim}$G$ is a function on $V$. \end{claim}
\begin{proof}
Clearly the domain of $G$ is $V$.  To see that $G$ is well-defined, we fix $Z$ and show that $G(Z)$ exists and is uniquely determined by $Z$.  Suppose there exists $\beta \in \lambda$ such that
\begin{eqnarray}\label{app1}
 Z:=(b_\gamma)_{\gamma \in \beta} \in M^\beta \textrm{~and~} M \vDash q_\beta(Z)
\end{eqnarray}
\noindent (otherwise, clearly $G(Z)$ is uniquely determined to be $\emptyset$).
Consider the set of formulas defined as follows:

\vspace{.1in}

\noindent $\displaystyle \pi(x_\beta) = \{\exists \ov{y} \bigwedge_{\varphi \in F} \varphi(X_\beta/Z) :  F=F(X_{\beta+1}; \ov{y})  \subset_\omega p, \; \ov{y} \in Y_{\beta+1}^n, \; n\in\omega \}.$

\vspace{.1in}

\noindent It is clear that $|Z \cup A| < \lambda$.

First, we argue that $\pi$ is a 1-type with respect to $\fMM$ over the set of parameters $Z \cup A$.  Fix any finite subset $H \subseteq \pi$ and define $m:=|H|$.  Then there exist finite sets of $\LL_A$-formulas
$F_t = \{\varphi^t_s(x_{\beta};\ov{y}_t;X_\beta) : s \in |F_t|\} \subset p$,
for some $\ov{y}_t \in Y_{\beta+1}^{n_t}$,
for all $t\in m$, such that:
$$H(x_\beta)=\{ \exists \ov{y}_t \; \bigwedge_{s \in |F_t|} \varphi^t_s(x_{\beta};\ov{y}_t;X_\beta/Z) : t\in m \}.$$

\noindent Let $\ov{y}$ be a tuple such that $\ran(\ov{y})=\bigcup_{t \in m} \ran(\ov{y}_t)$.  We may assume:

$$H(x_\beta)=\{ \exists \ov{y} \; \bigwedge_{s \in |F_t|} \varphi^t_s(x_{\beta};\ov{y};X_\beta/Z) : t\in m \}$$

\noindent For convenience, we may define $\psi_t :=\exists \ov{y} \; \bigwedge_{s \in |F_t|} \varphi^t_s(x_{\beta};\ov{y};X_\beta/Z)$ so that
$$H(x_\beta)=\{ \psi^t : t \in m\}$$

Define the formula 
$$\theta(X_{\beta}) := \exists x_\beta \exists \ov{y} \bigwedge_{t\in m} \bigwedge_{s \in  |F_t|} \varphi^t_s(x_{\beta};\ov{y};X_\beta).$$  
This  formula is in $q_\beta$ and so $\fMM \vDash \theta(X_\beta/Z)$ by \eqref{app1}.  Thus, there is some instantiation $b_\beta$ of $x_\beta$ such that
$$\fMM \vDash \exists \ov{y} \bigwedge_{t\in m} \bigwedge_{s \in |F_t|} \varphi^t_s(x_{\beta}/b_\beta;\ov{y};X_\beta/Z).$$

\noindent A fortiori, 
$$\fMM \vDash \bigwedge_{t\in m} \exists \ov{y}  \bigwedge_{s \in |F_t|} \varphi^t_s(x_{\beta}/b_\beta;\ov{y};X_\beta/Z),$$
i.e.,
$$\fMM \vDash \bigwedge_{t \in m} \psi^t(x_\beta/b_\beta),$$
so
$$\fMM \vDash \exists x_\beta \bigwedge_{\psi \in H} \psi(x_\beta).$$
Thus we have shown that $\pi$ is finitely satisfiable in $\fMM$.

\

Since we have established that $\pi$ is a type with respect to $\fMM$ over a parameter set of size less than $\lambda$, $\pi$ must be realized in $\fMM$, by $\lambda$-saturation.  Let $b \in M$ be the $\prec$-least element of $M$ that realizes $\pi$ in $\fMM$.  We will show a little more than is required by showing that $G(Z)$ is uniquely defined to be $b$.  

To see that $\fMM \vDash q_{\beta+1}(Z^\smf \la b \ra)$, consider any formula in $q_{\beta+1}(X_{\beta+1})$ such as
$$\exists \ov{y} \bigwedge_{\varphi \in F} \varphi(X_{\beta+1}; \ov{y})$$
where $F=F(X_{\beta+1},\ov{y}) \subset_\omega p$ and $\ov{y} \in Y_{\beta+1}^n$, for some $n\in\omega$.  For this same $F$:
$$\fMM \vDash \exists \ov{y} \bigwedge_{\varphi \in F} \varphi(X_{\beta+1}/Z^\smf\la b \ra; \ov{y})$$
since $b$ realizes $\pi$ in $\fMM$.
For any other element $b' \in M$ such that $\fMM \vDash q_{\beta+1}(Z^\smf \la b' \ra)$, this $b'$ realizes $\pi$, and so $b \prec b'$.  Thus $b$, as defined previously, is the $\prec$-least in $\fMM$ such that $\fMM \vDash q_{\beta+1}(Z^\smf \la b \ra)$.  This proves the claim that $G$ is a function.
\end{proof}

By Transfinite Recursion, there exists some $F: \textrm{Ord} \rightarrow V$ such that for all $\alpha$, 
\begin{eqnarray}\label{app2}
F(\alpha) = G(F \uphp \alpha).
\end{eqnarray}
Let $b_\alpha := F(\alpha)$, for any $\alpha$.  Equation \eqref{app2} may be rephrased as: $$b_\alpha = G((b_\beta)_{\beta \in \alpha} ).$$  Define $B_\alpha :=  (b_\beta)_{\beta \in \alpha}$, for any $\alpha \leq \lambda$.

\vspace{.1in}

\begin{claim} For each $\beta\leq \lambda$, $\fMM \vDash q_\beta(B_\beta)$.
\end{claim}

\begin{proof}  Let $C = \{ \beta \in \textrm{Ord} : \beta \leq \lambda \rightarrow \fMM \vDash q_\beta(B_\beta) \}$.  We will show that $C = \textrm{Ord}$ by transfinite induction.

Let $B_0 = \emptyset$ and $X_0 = \emptyset$.  Note that $q_0$ is the set of sentences that are existential quantifications of finite conjunctions from $p$.  Since $p$ is finitely satisfiable in $\fMM$, clearly $q_0$ is consistent with the theory of $\fMM$, and so we have shown $0 \in C$.

Suppose that $\beta \in C$ and assume that $\beta \in\lambda$, thus $\fMM \vDash q_\beta(B_\beta)$.  By the definition of $G$ in the transfinite recursion, $\fMM \vDash q_{\beta+1}(B_{\beta+1})$, and so $\beta+1 \in C$. (If $\beta = \lambda$, $\beta+1 \in C$ follows immediately.)

Let $\delta \neq 0$ be a limit ordinal and assume that $\gamma \in C$ for all $\gamma\in\delta$.  Assume that $\delta \leq \lambda$, otherwise the conclusion holds immediately.
To see that $\fMM \vDash q_\delta(X_\delta/B_\delta)$, fix any formula $\theta \in q_\delta$.  A formula is a finite string, so $\theta$ contains occurrences of only finitely many variables.  By a cofinality argument, $\theta \in q_\gamma$, for some $\gamma\in\delta$.
By the induction hypothesis, $\fMM \vDash \theta(X_\gamma/B_\gamma)$.
Since $B_\gamma$ is an initial segment of $B_\delta$, we have that $\fMM \vDash \theta(X_\delta/B_\delta)$.  Since $\theta$ was arbitrary, we have shown that $\fMM \vDash q_\delta(X_\delta/B_\delta)$ and so $\delta \in C$.
\end{proof}
Thus, we conclude that $\fMM \vDash q_\lambda(B_\lambda)$.  However, $q_\lambda = p(X)$, and so we have shown that 
$p$ is realized in $\fMM$, as desired.
$\eop_{\ref{technical2+}}$
\end{proof}

\section*{Acknowledgements} The authors thank M.~E.~Malliaris for the references from \cite{C-K} that are used
in the Appendix, as well as S.~Cramer for the idea of using a recursion on $\omega_1$ to exhaust the levels in $2^{<\omega_1}$ in the construction in Subsection \ref{cons_ce}.  We thank the anonymous referee for the detailed comments that improved the presentation of this paper.

\bibliographystyle{plain}

\normalsize
\baselineskip=17pt

\bibliography{bibmasterfive}

\end{document}


\title{Ultraproduct of finite linear orders}

\author{Dana Barto\v{s}ov\'a}
\address{University of Florida, 1400 Stadium Road, Gainesville, FL 32601, USA}
\email{dbartosova@ufl.edu}
\urladdr{https://people.clas.ufl.edu/dbartosova/}

\author{Mirna D\v zamonja}
\address{base-name Logique Consult, 7 Impasse Charti\`{e}re, 75005 Paris, France}
\email{logique.consult@gmail.com}
\urladdr{https://www.logiqueconsult.eu}
 
\author{Rehana Patel}
\address{Wesleyan University, Science Tower 655, 265 Church Street, Middletown, CT 06459-0128, USA}
\email{rpatel@wesleyan.edu}
\urladdr{}

\author{Lynn Scow}
\address{California State University, San Bernardino
5500 University Parkway 
San Bernardino, CA 92407, USA} 
\email{lscow@csusb.edu} 
\urladdr{https://www.math.csusb.edu/faculty/lscow/lscow.html}

\thanks{We thank Jouko Väänänen for showing us the main part of the enclosed proof in his e-mail dated 07 June 2022. We have learnt the fact
from \S2 of his paper \cite{zbMATH02208072}, where it is given as an example. Even though he generously never mentioned not being cited properly, we use the occasion to apologise to him for not having mentioned these references in the published version of our paper \cite{BigRamsey}.}

\subjclass[2010]{03C20, 05D10, 03C50, 03C13, 05C55}

\keywords{ultraproduct, linear order, pseudo-finite}

\begin{abstract}
The note contains a proof on the structure of the ultraproduct of finite linear orders, a theorem we gave without a proof at the beginning of \S6 in \cite{BigRamsey}.
\end{abstract}

\maketitle


\setcounter{page}{1}
\thispagestyle{empty}

\begin{small}
\renewcommand\contentsname{\!\!\!\!}
\setcounter{tocdepth}{3}
\tableofcontents
\end{small}

\newpage


\section{The ultraproduct of finite linear orders}\label{th}
In \S6 of \cite{BigRamsey} we state the following Theorem \ref{th}, saying that the fact appears in many resources and quoting the notes by Dario Garc{\'i}a \cite{Garcia} for a proof.
Modulo a few details left to the reader, these notes, intended for teaching, indeed contain a proof of a large part of the theorem, given there as Lemma 1.1.7. However, the first proof we learnt was from Jouko Väänänen as explained in the abstract, and it contains the proof of the rest of the statement as well. We write this addendum of our paper by combining the two approaches and adding a few details. Here $\omega^\ast$ stands for the order type of the reverse order
of $\omega$.

\begin{theorem}\label{th} All infinite ultraproducts of a countable sequence of finite linear orders are elementarily equivalent and are of the form
	\begin{eqnarray}\label{ult1}
	\omega+\sum_{\mathfrak{L}}(\omega^*+\omega)+\omega^*,
	\end{eqnarray}
	where $\omega^*$ is the reverse of the natural linear order on $\omega$ and $\mathfrak{L}$ is an $\aleph_1$-saturated dense linear order without endpoints.
\end{theorem}

\begin{proof} Let us assume that we have an ultraproduct ${\mathfrak M}=\Pi_{i<\omega} L_i/U$, where $U$ is an ultrafiler over $\omega$, each $L_i$ is a finite linear order, while ${\mathfrak M}$ is infinite.

We first show that ${\mathfrak M}$ is $\aleph_1$-saturated.
Theorem 6.1.1 from \cite{ChKe} states that an ultraproduct over a countably incomplete filter of a family of models of a countable language
is $\aleph_1$-saturated (this also appears a Corollary 2.24 of our paper  \cite{BigRamsey}). In the case of ultrafilters over $\omega$, if we have a countably complete ultrafilter then it will have a non-empty intersection. If 
$U$ were to be such an ultrafilter, its intersection would contain a point $n$. In particular, every two functions in the ultraproduct that have the same value at $n$ are ultraproduct equivalent. It follows that the cardinality of the ultraproduct is that of $L_n$, so finite - in contradiction with our assumptions.
Hence ${\mathfrak M}$ is $\aleph_1$-saturated.

Now we use \L o\'s' theorem several times. First, we note that by that theorem, ${\mathfrak M}$ is a linear order, which we shall denote by $<$. Each $L_i$ has the first and the last element, hence by  \L o\'s' theorem ${\mathfrak M}$ also has the first and the last element. Then, in each $L_i$ every element $a$  except the last one has an immediate successor $S(a)$ and each element $a$ except the first one has an immediate predecessor $P(a)$. So again by \L o\'s' theorem, the same is true in ${\mathfrak M}$ . This implies that ${\mathfrak M}$ has an initial part isomorphic to
$\omega$ and the final part isomorphic to $\omega^\ast$. Say  ${\mathfrak M}=\omega+ B+\omega^\ast$, where we use $B$ for ``between". 

Note that we have also proven that ${\mathfrak M}$ has definable functions $S$ and $P$, where $S$ is defined on all but the last element and $P$ on all but the first element.

Next is to define an equivalence relation on $B$ by saying $a\sim b$ if there is a finite composition $f$ of the functions $S$ and $P$ such that 
$f(a)=b$. It is easily checked that this is indeed an equivalence relation and that the equivalence classes $[a]$ have the order type of $\mathbb Z$, which is
$\omega^\ast+\omega$. The factor order $\mathfrak{L}=</\sim$ is a linear order that now gives us ${\mathfrak M}=\omega+\sum_{\mathfrak{L}}(\omega^*+\omega)+\omega^*$. 

Now we prove that $\mathfrak{L}$ has no end points and, then, that it is $\aleph_1$-saturated. Suppose for a contradiction that $[a]$ were to be a minimal element for 
$\mathfrak{L}$ and let $\langle a_n:\,n<\omega\rangle$ be the initial part of ${\mathfrak M}$ of order type $\omega$. We write $[a]=\{a+k:\,k\in \mathbb Z\}$,
with $+ k$ for $k\in \mathbb Z$ being defined from $S$ and $P$ in the obvious way.
Consider the type
\[
\Gamma(x; \langle a_n:\,n<\omega\rangle, \langle a+k:\,k\in \mathbb Z\rangle  )\equiv \{a_n<x:\,n<\omega\}\cup \{x<a+k:\,k\in \mathbb Z\}.
\]
Since this type is finitely consistent with ${\mathfrak M}$, it follows by the compactness theorem that it is consistent with ${\mathfrak M}$. Since it has countably many parameters from ${\mathfrak M}$, it follows by the $\aleph_1$-saturation of ${\mathfrak M}$ that $\Gamma(x; \langle a_n:\,n<\omega\rangle, 
\langle a+k:\,k\in \mathbb Z\rangle)$ is realised in ${\mathfrak M}$, which implies that there is an element $b$ of ${\mathfrak M}$ strictly above all $a_n$ and strictly below all $a+k$. Therefore $b\in B$ and we obtain that $b+m<a+k$ for all $m,k$, hence $[b]<_{\mathcal L} [a]$ -a contradiction with 
$[a]$ being the minimal element of ${\mathcal L}$. Therefore ${\mathcal L}$ has no minimal element and it similarly has no maximal element. In a similar way we can show that ${\mathcal L}$ is a dense linear order. This is weaker than being $\aleph_1$-saturated, but is needed for our proof of $\aleph_1$-saturation, 
as follows.

Namely, by the proofs pg. 51-pg. 54 of \cite{ChKe}, the theory DLO of a dense linear order with no end points, admits the elimination of quantifiers. Therefore, to check that its model ${\mathcal L}$ is $\aleph_1$-saturated, we only need to check that every consistent countable cut in it is filled. That is, for every two sequences $\langle [b]_n:\,n<\omega\rangle$ increasing in $\mathfrak{L}$ and $\langle [c]_n:\,n<\omega\rangle$ decreasing in $\mathfrak{L}$ such that $[b]_n<_{\mathcal L} [c]_m$ for every $n,m$, we must find an element $[d]\in \mathfrak{L}$ such that $[b]_n<_\mathfrak{L} [d]<_\mathfrak{L} [c]_m$ for every
$n,m$. Similarly to the above with $\Gamma$, we translate this into a consistent type being filled in the $\aleph_1$-saturated model ${\mathfrak M}$. In this
way we obtain an element $d$ of ${\mathfrak M}$ such that $b_n<d<c_m$ for all $n,m$. However, by the structure of $B$ and the assumptions on 
$[b]_n$ and $[c]_m$, it follows that in fact the whole equivalence class $[d]$ satisfies $[b]_n<_\mathfrak{L} [d]<_\mathfrak{L} [c]_m$ for every
$n,m$, as required.

The final part of the proof is that all such models $\mathfrak M$ are elementary equivalent. Let $\mathfrak M'$ be another such model, so it is of the form 
${\mathfrak M'}=\omega+\sum_{\mathfrak{L'}}(\omega^*+\omega)+\omega^*$, with the order that we denote by $<'$. Let $\varphi$ be a sentence true in 
$\mathfrak M$. Let $n$ be the quantifier rank of $\varphi$. By Corollary 1.5.5 of \cite{ChKe}, which is the fine tuning of the elimination of quantifiers for DLO,  we obtain that there is a quantifier free formula $\psi(x_0, \ldots x_{n-1}, y_0, \ldots y_{n-1})$ which is a Boolean combination of the formulas of the type
$x_i\le y_j$ and which is DLO-equivalent to $\varphi$. By Lemma 1 in \cite{zbMATH02208072}, which gives a characterisation of such structures 
$\mathfrak M$, being an ultraproduct of finite structures is what is known as pseudo-finite. By the same Lemma, there is a finite linear order $L$ in which 
$\psi$ is true, as exemplified by some $\langle a_0, \ldots a_{n-1}, b_0, \ldots b_{n-1}\rangle$.

Since $\mathfrak{M'}$ is $\aleph_1$-saturated, it is $\aleph_2$-universal
(Theorem 5.1.12 \cite{ChKe}), so in particular there is an embedding $g$ of $L$ into $\mathfrak{M'}$. Hence we have that in $\mathfrak{M'}$ the sentence $\psi$, which is just an arrangement of relative orderings of pairs of variables, is true of $\langle g(a_0), \ldots g(a_{n-1}), g(b_0), \ldots g(b_{n-1})\rangle$. Since $\mathfrak{M'}$
is a model of DLO, we have by the choice of $\psi$ that $\mathfrak{M'}$ satisfies $\varphi$.
$\eop_{\ref{th}}$
\end{proof}

\begin{remark}\label{CH} (1) Since ${\mathfrak M}$ from Theorem \ref{th} is $\aleph_1$-saturated, it follows that it has cardinality $2^{\aleph_0}$. The proof
is the same as in the above: repeatedly using the fact that ${\mathfrak M}$ is a dense linear order (since it is moreover $\aleph_1$-saturated), we can construct a copy of $2^{<\omega}$ within ${\mathfrak M}$. Then all its cuts must be filled in ${\mathfrak M}$, by $\aleph_1$-saturation. Hence, the cardinality of ${\mathfrak M}$ is at least $2^{\aleph_0}$. Since the
cardinality of the full product $\Pi_{i<\omega} L_i$ is at most that of ${}^{\omega}{\omega}$, which is $2^{\aleph_0}$, we obtain the upper bound of 
 $2^{\aleph_0}$ on the cardinality of ${\mathfrak M}$. Hence $|{\mathfrak M}|=2^{\aleph_0}$. This is a special case of a theorem by T. Frayne, A.C. Morel and D.S Scott from \cite{FMScott}, which in \cite{BigRamsey} we cited as Theorem 2.25.

It follows from the above that under CH, which is the situation considered in  \cite{BigRamsey}, the ultraproduct ${\mathfrak M}$ is fully saturated, and hence unique up to isomorphism (Theorem 5.17 of \cite{ChKe}, using Proposition 5.1.6.(i) of \cite{ChKe} which states that every saturated model is special).

{\noindent (2)} To prove the universality of a saturated model one uses the classical Ehrenfeucht-Fraïssé game of length $\omega$. In the argument in the last paragraph of the proof of Theorem \ref{th} we have only used that it was possible to embed the finite order $L$ into $\mathfrak{M'}$, so a game of finite length suffices. This is the proposition of the proof sketched by Väänänen in his mail, using Proposition 5.2.5 of \cite{joukobookgames} in place of Theorem 5.1.12 of \cite{ChKe}.
\end{remark}

\bibliographystyle{plainurl}

\normalsize
\baselineskip=17pt

\bibliography{../../bibliomaster}